\documentclass[12pt,amstex]{amsart}

\usepackage{epsfig}
\usepackage{amsmath}
\usepackage{amssymb}
\usepackage{amscd}
\usepackage{graphicx}

\usepackage{pstricks}

\input xy
\xyoption{all}

\newtheorem{thm}{Theorem}[section]
\newtheorem{lem}[thm]{Lemma}

\newtheorem{prop}[thm]{Proposition}

\newtheorem{cor}[thm]{Corollary}
\theoremstyle{definition}
\newtheorem{de}[thm]{Definition}

\theoremstyle{remark}
\newtheorem{rem}[thm]{Remark}

\numberwithin{equation}{section}

\def \N {\mathbb N}

\def \Z {\mathbb Z}

\def \F {\mathcal F}
\def \G {\mathcal{G}}
\def \U {\mathcal U}

\def \X {\mathcal{X}}

\def \O {\mathcal{O}}
\def \t {\mathcal{T}}

\def \Q {{\bf Q}}
\def \RP {{\bf RP}}
\def \M {{\bf M}}

\def \int {{\rm int}}
\def \cl {{\rm cl}}

\def \id {{\rm id}}

\def \a {\alpha }
\def \b {\beta}
\def \ep {\epsilon}
\def \d {\delta}
\def \D {\Delta}
\def \c {\circ}

\def\w {\omega}
\def \ov {\overline}
\def \lra{\longrightarrow}

\begin{document}
\title{Regionally proximal relation of order $d$ is an equivalence one for minimal systems
and a combinatorial consequence}

\author{Song Shao}
\author{Xiangdong Ye}

\address{Department of Mathematics, University of Science and Technology of China,
Hefei, Anhui, 230026, P.R. China.}

\email{songshao@ustc.edu.cn} \email{yexd@ustc.edu.cn}

\subjclass[2000]{Primary: 37B05; 37A99} \keywords{Nilsystems;
regionally proximal relation; minimal systems}

\thanks{S. Shao is supported by NNSF of
China (10871186), and X.D. Ye is supported by NNSF of China
(10531010, 11071231) and 973 programm.}

\date{}

\begin{abstract}
By proving the minimality of face transformations acting on the
diagonal points and searching the points allowed in the minimal
sets, it is shown that the regionally proximal relation of order
$d$, $\RP^{[d]}$, is an equivalence relation for minimal systems.
Moreover, the lifting of $\RP^{[d]}$ between two minimal systems is
obtained, which implies that the factor induced by $\RP^{[d]}$ is
the maximal $d$-step nilfactor. The above results extend the same
conclusions proved by Host, Kra and Maass for minimal distal
systems.

A combinatorial consequence is that if $S$ is a dynamically syndetic
subset of $\Z$, then for each $d\ge 1$, $$\{(n_1,\ldots,n_d)\in
\Z^d: n_1\ep_1+\cdots +n_d\ep_d\in S, \ep_i\in \{0,1\}, 1\le i\le
d\}$$ is syndetic. In some sense this is the topological
correspondence of the result obtained by Host and Kra for positive
upper Banach density subsets using ergodic methods.
\end{abstract}

\maketitle

\markboth{Regionally proximal relation of order $d$}{S. Shao and
X.D. Ye}

\section{Introduction}

The background of our study can be seen both in ergodic theory and
topological dynamics.

\subsection{Background in ergodic theory}
The connection between ergodic theory and additive combinatorics was
built in the 1970's with Furstenberg's beautiful proof of
Szemer\'edi's theorem via ergodic theory \cite{F77}. Furstenberg's
proof paved the way for obtaining new combinatorial results using
ergodic methods, as well as leading to numerous developments within
ergodic theory. Roughly speaking, Furstenberg \cite{F77} proved
Szemer\'edi's theorem via the following ergodic theorem: let $T$ be
a measure-preserving transformation on the probability space
$(X,\mathcal{B},\mu)$, and let $A\in \mathcal{B}$ with positive
measure. Then for every integer $d \ge 1$,
\begin{equation*}
    \liminf_{N\to \infty} \frac{1}{N}\sum_{n=0}^{N-1}
    \mu(A\cap T^{-n}A\cap T^{-2n}A\cap \ldots \cap T^{-dn}A)>0.
\end{equation*}

So it is natural to ask about the convergence of these averages, or
more generally about the convergence in $L^2(X,\mu)$ of the multiple
ergodic averages $$ \frac 1 N\sum_{n=0}^{N-1}f_1(T^nx)\ldots
f_d(T^{dn}x) ,$$ where $f_1, \ldots , f_d \in L^\infty(X,\mu)$.
After nearly 30 years' efforts of many researchers, this problem was
finally solved in \cite{HK05, Z}.

In their proofs the notion of characteristic factors plays a great
role. Let us see why this notion is important. Loosely speaking, the
structure theorem of \cite{HK05, Z} states that if one wants to
understand the multiple ergodic averages
$$ \frac 1 N\sum_{n=0}^{N-1}f_1(T^nx)\ldots f_d(T^{dn}x) ,$$
one can replace each function $f_i$ by its conditional expectation
on some $d$-step nilsystem ($1$-step nilsystem is the Kroneker's
one). Thus one can reduce the problem to the study of the same
average in a nilsystem, i.e. reducing the average in an arbitrary
system to a more tractable question.
Note that the multiple ergodic average for commuting transformations
was obtained by Tao \cite{Tao} using finitary ergodic method, see
\cite{Austin,H} for more traditional ergodic proofs. Unfortunately,
in this more general setting, the characteristic factors are not
known up till now.

\medskip

In \cite{HK05}, some useful tools, such as dynamical
parallelepipeds, ergodic uniformity seminorms etc., were introduced
in the study of dynamical systems. Their further applications were
discussed in \cite{H, HK08, HK09, HKM, HM}. Now a natural and
important question is what the topological correspondence of
characteristic factors is. The history how to obtain the topological
counterpart of characteristic factors will be discussed in the next
subsection.

\subsection{Background in topological dynamics}
In some sense an equicontinuous system is the simplest system in
topological dynamics. In the study of topological dynamics, one of
the first problems was to characterize the equicontinuous structure
relation $S_{eq}(X)$ of a system $(X, T)$; i.e. to find the smallest
closed invariant equivalence relation $R(X)$ on $(X, T)$ such that
$(X/ R(X), T)$ is equicontinuous. A natural candidate for $R(X)$ is
the so-called regionally proximal relation $\RP(X)$ \cite{EG}. By
the definition, $\RP(X)$ is closed, invariant, and reflexive, but
not necessarily transitive. The problem was then to find conditions
under which $\RP(X)$ is an equivalence relation. It turns out to be
a difficult problem. Starting with Veech \cite{V68}, various
authors, including MacMahon \cite{Mc}, Ellis-Keynes \cite{EK}, came
up with various sufficient conditions for $\RP(X)$ to be an
equivalence relation. For somewhat different approach, see
\cite{AG}. Note that in our case, $T: X\rightarrow X$ being
homeomorphism and $(X,T)$ being minimal, $\RP(X)$ is always an
equivalence relation.

In \cite{HM} Host and Maass tried to find the topological
counterpart of characteristic factors and obtained their goal in
some particular case. Recently, Host, Kra and Maass \cite{HKM}
continued the study and succeeded for all minimal distal systems,
which can be viewed as an analog of the purely ergodic structure
theory of \cite{F77, HK05, Z}. Note that previously the counterpart
of characteristic factors in topological dynamics was studied by
Glasner \cite{G93,G94}, where he considered the characterization of
nil-system of order 2 in \cite{G93} and studied the characteristic
factors for the action $T\times T^2\times \ldots \times T^n$ in
\cite{G94}. To get the characteristic factors in topological
dynamics, in \cite{HKM, HM}, for distal minimal systems a certain
generalization of the regionally proximal relation is used to
produce the maximal nilfactors.

\medskip

Here is the notion of the regionally proximal relation of order $d$
defined in \cite{HM,HKM}.

\begin{de}
Let $(X, T)$ be a system and let $d\ge 1$ be an integer. A pair $(x,
y) \in X\times X$ is said to be {\em regionally proximal of order
$d$} if for any $\d  > 0$, there exist $x', y'\in X$ and a vector
${\bf n} = (n_1,\ldots , n_d)\in\Z^d$ such that $\rho(x, x') < \d,
\rho(y, y') <\d$, and $$ \rho(T^{{\bf n}\cdot \ep}x', T^{{\bf
n}\cdot \ep}y') < \d\ \text{for any $\ep\in \{0,1\}^d$,
$\ep\not=(0,\ldots,0)$},
$$ where ${\bf n}\cdot \ep = \sum_{i=1}^d \ep_in_i$. The set of
regionally proximal pairs of order $d$ is denoted by $\RP^{[d]}(X)$,
which is called {\em the regionally proximal relation of order $d$}.
\end{de}

It is easy to see that $\RP^{[d]}(X)$ is a closed and invariant
relation for all $d\in \N$. When $d=1$, $\RP^{[d]}(X)$ is nothing
but the classical regionally proximal relation. In \cite{HKM}, for
distal minimal systems the authors showed that $\RP^{[d]}(X)$ is a
closed invariant equivalence relation, and the quotient of $X$ under
this relation is its maximal $d$-step nilfactor. So it remains the
question open: is $\RP^{[d]}(X)$ an equivalence relation for any
minimal system? The purpose of the current paper is to settle down
the question.


\subsection{Main results}

In this article, we show that for each minimal system $\RP^{[d]}(X)$
is a closed invariant equivalence relation and the quotient of $X$
under this relation is its maximal $d$-step nilfactor.

Note that a subset $S$ of $\Z$ is {\it dynamically syndetic} if
there is a minimal system $(X,T)$, $x\in X$ and an open neighborhood
$U$ of $x$ such that $S=\{n\in \Z: T^nx\in U\}$. Equivalently,
$S\subset \Z$ is dynamically syndetic if and only if $S$ contains
$\{0\}$ and $1_S$ is a minimal point of $(\{0,1\}^\Z, \sigma )$,
where $\sigma$ is the shift map. A subset $S$ of $\Z^d$ is {\it
syndetic} if there exists a finite subset $F\subset \Z^d$ such that
$S+F=\Z^d$. A combinatorial consequence of our results is that if
$S$ is a dynamically syndetic subset of $\Z$, then for each $d\ge
1$,
$$\{{\bf n}=(n_1,\ldots,n_d)\in \Z^d: \ep\cdot {\bf n}=
n_1\ep_1+\cdots +n_d\ep_d\in S,
\ep_i\in \{0,1\}, 1\le i\le d\}$$ is syndetic. In some sense this is
the topological correspondence of the following result obtained by
Host and Kra for positive upper Banach density subsets using ergodic
methods.

\begin{thm}\cite[Theorem 1.5]{HK05}
Let $A\subset \Z$ with $\overline{d}(A)> \d >0$ and let $d\in \N$,
then
$$\{{\bf n}=(n_1,n_2,\ldots, n_k)\in \Z^d :\overline{d}
\Big(\bigcap _{\ep\in \{0,1\}^d}(A+\ep\cdot {\bf n})\Big)\ge
\d^{2^d}\}$$ is syndetic, where $\overline{d}(B)$ denotes the upper
density of $B\subset \Z$.
\end{thm}

In \cite{HKM} the authors showed that the regionally proximal
relation of order $d$ is an equivalence relation for minimal distal
systems without using the enveloping semigroup theory except one
known result that the distal extension between minimal systems is
open (which is proved using the theory). In our situation we are
forced to use the theory. The main idea of the proof is the
following. First using the structure theory of a minimal system we
show that the face transformations acting on the diagonal points are
minimal, and then we prove some equivalence conditions for a pair
being regionally proximal of order $d$. A key lemma here is to
switch from a cubic point to a face point. Combining the minimality
with the conditions we show that the regionally proximal relation of
order $d$ is an equivalence relation for minimal systems. Finally we
show that $\RP^{[d]}$ can be lifted up from a factor to an extension
between two minimal systems, which implies that the factor induced
by $\RP^{[d]}$ is the maximal $d$-step nilfactor. It will be nice if
one could have a proof of the minimality of face actions on the
diagonal points without using the structure theory of minimal flows.

We remark that many results of the paper can be extended to abelian
group actions.

\subsection{Organization of the paper}
In Section \ref{section-pre}, we introduce some basic notions used
in the paper. Since we will use tools from abstract topological
dynamics, we collect basic facts about it in Appendix
\ref{section-tds}. In Section \ref{section-main}, main results of
the paper are discussed. The three sections followed are devoted to
give proofs of main results. Note that lots of results obtained
there have their independent interest. In the final section some
applications are given.

\subsection{Thanks}

We thank V. Bergelson, E. Glasner, W. Huang, H.F. Li, A. Maass for
helpful discussions. Particularly we thank E. Glasner for sending us
his note on the topic, and H.F. Li for the very careful reading
which helps us correct misprints and simplify some proofs.

\section{Preliminaries}\label{section-pre}

In this section we introduce notions about dynamical parallelepipeds
and nilsystems etc.. For more details see \cite{HK05, HKM}.

\subsection{Topological dynamical systems}

A {\em transformation} of a compact metric space X is a
homeomorphism of X to itself. A {\em topological dynamical system},
referred to more succinctly as just a {\em system}, is a pair $(X,
T)$, where $X$ is a compact metric space and $T : X \rightarrow  X$
is a transformation. We use $\rho(\cdot, \cdot)$ to denote the
metric in $X$. We also make use of a more general definition of a
topological system. That is, instead of just a single transformation
$T$, we will consider a countable abelian group of transformations.
We collect basic facts about topological dynamics under general
group actions in Appendix \ref{section-tds}.

A system $(X, T)$ is {\em transitive} if there exists some point
$x\in X$ whose orbit $\O(x,T)=\{T^nx: n\in \Z\}$ is dense in $X$ and
we call such a point a {\em transitive point}. The system is {\em
minimal} if the orbit of any point is dense in $X$. This property is
equivalent to say that X and the empty set are the only closed
invariant sets in $X$.


\subsection{Cubes and faces}

Let $X$ be a set, let $d\ge 1$ be an integer, and write $[d] = \{1,
2,\ldots , d\}$. We view $\{0, 1\}^d$ in one of two ways, either as
a sequence $\ep=(\ep_1,\ldots ,\ep_d)$ of $0'$s and $1'$s; or as a
subset of $[d]$. A subset $\ep$ corresponds to the sequence
$(\ep_1,\ldots, \ep_d)\in \{0,1\}^d$ such that $i\in \ep$ if and
only if $\ep_i = 1$ for $i\in [d]$. For example, ${\bf
0}=(0,0,\ldots,0)\in \{0,1\}^d$ is the same to $\emptyset \subset
[d]$.

If ${\bf n} = (n_1,\ldots, n_d)\in \Z^d$ and $\ep\in \{0,1\}^d$, we
define
$${\bf n}\cdot \ep = \sum_{i=1}^d n_i\ep_i .$$
If we consider $\ep$ as $\ep\subset [d]$, then ${\bf n}\cdot \ep  =
\sum_{i\in \ep} n_i .$

\medskip

We denote $X^{2^d}$ by $X^{[d]}$. A point ${\bf x}\in X^{[d]}$ can
be written in one of two equivalent ways, depending on the context:
$${\bf x} = (x_\ep :\ep\in \{0,1\}^d )= (x_\ep : \ep\subset [d]). $$
Hence $x_\emptyset =x_{\bf 0}$ is the first coordinate of ${\bf x}$.
As examples, points in $X^{[2]}$ are like
$$(x_{00},x_{10},x_{01},x_{11})=(x_{\emptyset}, x_{\{1\}},x_{\{2\}},x_{\{1,2\}}),$$
and points in $X^{[3]}$ are like
\begin{eqnarray*}
 & & (x_{000},x_{100},x_{010},x_{110},x_{001},x_{101},x_{011},x_{111}) \\
     &=&(x_{\emptyset}, x_{\{1\}},x_{\{2\}},x_{\{1,2\}},x_{\{3\}}, x_{\{1,3\}},x_{\{2,3\}},x_{\{1,2,3\}}).
                                \end{eqnarray*}

For $x \in X$, we write $x^{[d]} = (x, x,\ldots , x)\in  X^{[d]}$.
The diagonal of $X^{[d]}$ is $\D^{[d]} = \{x^{[d]}: x\in X\}$.
Usually, when $d=1$, denote the diagonal by $\D_X$ or $\D$ instead
of $\D^{[1]}$.

A point ${\bf x} \in X^{[d]}$ can be decomposed as ${\bf x} = ({\bf
x'},{\bf  x''})$ with ${\bf x}', {\bf x}''\in X^{[d-1]}$, where
${\bf x}' = (x_{\ep0} : \ep\in \{0,1\}^{d-1})$ and ${\bf x}''=
(x_{\ep1} : \ep\in \{0,1\}^{d-1})$. We can also isolate the first
coordinate, writing $X^{[d]}_* = X^{2^d-1}$ and then writing a point
${\bf x}\in X^{[d]}$ as ${\bf x} = (x_\emptyset, {\bf x}_*)$, where
${\bf x}_*= (x_\ep : \ep\neq \emptyset) \in X^{[d]}_*$.

\medskip

Identifying $\{0, 1\}^d$ with the set of vertices of the Euclidean
unit cube, a Euclidean isometry of the unit cube permutes the
vertices of the cube and thus the coordinates of a point $x\in
X^{[d]}$. These permutations are the {\em Euclidean permutations} of
$X^{[d]}$. For details see \cite{HK05}.

\subsection{Dynamical parallelepipeds}
\begin{de}
Let $(X, T)$ be a topological dynamical system and let $d\ge 1$ be
an integer. We define $\Q^{[d]}(X)$ to be the closure in $X^{[d]}$
of elements of the form $$(T^{{\bf n}\cdot \ep}x=T^{n_1\ep_1+\ldots
+ n_d\ep_d}x: \ep= (\ep_1,\ldots,\ep_d)\in\{0,1\}^d) ,$$ where ${\bf
n} = (n_1,\ldots , n_d)\in \Z^d$ and $ x\in X$. When there is no
ambiguity, we write $\Q^{[d]}$ instead of $\Q^{[d]}(X)$. An element
of $\Q^{[d]}(X)$ is called a (dynamical) {\em parallelepiped of
dimension $d$}.
\end{de}

It is important to note that $\Q^{[d]}$ is invariant under the
Euclidean permutations of $X^{[d]}$.

As examples, $\Q^{[2]}$ is the closure in $X^{[2]}=X^4$ of the set
$$\{(x, T^mx, T^nx, T^{n+m}x) : x \in X, m, n\in \Z\}$$ and $\Q^{[3]}$
is the closure in $X^{[3]}=X^8$ of the set $$\{(x, T^mx, T^nx,
T^{m+n}x, T^px, T^{m+p}x, T^{n+p}x, T^{m+n+p}x) : x\in X, m, n, p\in
\Z\}.$$

\begin{de}
Let $\phi: X\rightarrow Y$ and $d\in \N$. Define $\phi^{[d]}:
X^{[d]}\rightarrow Y^{[d]}$ by $(\phi^{[d]}{\bf x})_\ep=\phi x_\ep$
for every ${\bf x}\in X^{[d]}$ and every $\ep\subset [d]$.

Let $(X, T)$ be a system and $d\ge 1$ be an integer. The {\em
diagonal transformation} of $X^{[d]}$ is the map $T^{[d]}$.
\end{de}

\begin{de}
{\em Face transformations} are defined inductively as follows: Let
$T^{[0]}=T$, $T^{[1]}_1=\id \times T$. If
$\{T^{[d-1]}_j\}_{j=1}^{d-1}$ is defined already, then set
$$T^{[d]}_j=T^{[d-1]}_j\times T^{[d-1]}_j, \ j\in \{1,2,\ldots, d-1\},$$
$$T^{[d]}_d=\id ^{[d-1]}\times T^{[d-1]}.$$
\end{de}

It is easy to see that for $j\in [d]$, the face transformation
$T^{[d]}_j : X^{[d]}\rightarrow X^{[d]}$ can be defined by, for
every ${\bf x} \in X^{[d]}$ and $\ep\subset [d]$,$$T^{[d]}_j{\bf x}=
\left\{
  \begin{array}{ll}
    (T^{[d]}_j{\bf x})_\ep=Tx_\ep, & \hbox{$j\in \ep$;} \\
    (T^{[d]}_j{\bf x})_\ep=x_\ep, & \hbox{$j\not \in \ep$.}
  \end{array}
\right.$$

The {\em face group} of dimension $d$ is the group $\F^{[d]}(X)$ of
transformations of $X^{[d]}$ spanned by the face transformations.
The {\em parallelepiped group} of dimension $d$ is the group
$\G^{[d]}(X)$ spanned by the diagonal transformation and the face
transformations. We often write $\F^{[d]}$ and $\G^{[d]}$ instead of
$\F^{[d]}(X)$ and $\G^{[d]}(X)$, respectively. For $\G^{[d]}$ and
$\F^{[d]}$, we use similar notations to that used for $X^{[d]}$:
namely, an element of either of these groups is written as $S =
(S_\ep : \ep\in\{0,1\}^d)$. In particular, $\F^{[d]} =\{S\in
\G^{[d]}: S_\emptyset ={\rm id}\}$.

\medskip

For convenience, we denote the orbit closure of ${\bf x}\in X^{[d]}$
under $\F^{[d]}$ by $\overline{\F^{[d]}}({\bf x})$, instead of
$\overline{\O({\bf x}, \F^{[d]})}$.

It is easy to verify that $\Q^{[d]}$ is the closure in $X^{[d]}$ of
$$\{Sx^{[d]} : S\in \F^{[d]}, x\in X\}.$$

If $x$ is a transitive point of $X$, then $\Q^{[d]}$ is the closed
orbit of $x^{[d]}$ under the group $\G^{[d]}$.

\subsection{Nilmanifolds and nilsystems}

Let $G$ be a group. For $g, h\in G$, we write $[g, h] =
ghg^{-1}h^{-1}$ for the commutator of $g$ and $h$ and we write
$[A,B]$ for the subgroup spanned by $\{[a, b] : a \in A, b\in B\}$.
The commutator subgroups $G_j$, $j\ge 1$, are defined inductively by
setting $G_1 = G$ and $G_{j+1} = [G_j ,G]$. Let $k \ge 1$ be an
integer. We say that $G$ is {\em $k$-step nilpotent} if $G_{k+1}$ is
the trivial subgroup.

\medskip

Let $G$ be a $k$-step nilpotent Lie group and $\Gamma$ a discrete
cocompact subgroup of $G$. The compact manifold $X = G/\Gamma$ is
called a {\em $k$-step nilmanifold}. The group $G$ acts on $X$ by
left translations and we write this action as $(g, x)\mapsto gx$.
The Haar measure $\mu$ of $X$ is the unique probability measure on
$X$ invariant under this action. Let $\tau\in G$ and $T$ be the
transformation $x\mapsto \tau x$ of $X$. Then $(X, T, \mu)$ is
called a {\em basic $k$-step nilsystem}. When the measure is not
needed for results, we omit and write that $(X, T)$ is a basic
$k$-step nilsystem.

\medskip

We also make use of inverse limits of nilsystems and so we recall
the definition of an inverse limit of systems (restricting ourselves
to the case of sequential inverse limits). If $(X_i,T_i)_{i\in \N}$
are systems with $diam(X_i)\le M<\infty$ and $\phi_i:
X_{i+1}\rightarrow X_i$ are factor maps, the {\em inverse limit} of
the systems is defined to be the compact subset of $\prod_{i\in
\N}X_i$ given by $\{ (x_i)_{i\in \N }: \phi_i(x_{i+1}) = x_i,
i\in\N\}$, which is denoted by $\displaystyle
\lim_{\longleftarrow}\{X_i\}_{i\in \N}$. It is a compact metric
space endowed with the distance $\rho(x, y) =\sum_{i\in \N} 1/2^i
d_i(x_i, y_i )$. We note that the maps $\{T_i\}$ induce a
transformation $T$ on the inverse limit.

\begin{thm}[Host-Kra-Maass]\cite[Theorem 1.2]{HKM}\label{HKM}
Assume that $(X, T)$ is a transitive topological dynamical system
and let $d \ge 2$ be an integer. The following properties are
equivalent:
\begin{enumerate}
  \item If $x, y \in \Q^{[d]}(X)$ have $2^d-1$ coordinates in common, then $x =
y$.
  \item If $x, y \in X$ are such that $(x, y,\ldots , y) \in  \Q^{[d]}(X)$,
then $x = y$.
  \item $X$ is an inverse limit of basic $(d - 1)$-step minimal
nilsystems.
\end{enumerate}

A transitive system satisfying either of the equivalent properties
above is called a {\em $(d-1)$-step nilsystem} or a {\em system of
order $(d-1)$}.
\end{thm}

\subsection{Definition of the regionally proximal relations}
\begin{de}
Let $(X, T)$ be a system and let $d\ge 1$ be an integer. A pair $(x,
y) \in X\times X$ is said to be {\em regionally proximal of order
$d$} if for any $\d  > 0$, there exist $x', y'\in X$ and a vector
${\bf n} = (n_1,\ldots , n_d)\in\Z^d$ such that $\rho(x, x') < \d,
\rho(y, y') <\d$, and $$ \rho(T^{{\bf n}\cdot \ep}x', T^{{\bf
n}\cdot \ep}y') < \d\ \text{for any nonempty $\ep\subset [d]$}.$$
(In other words, there exists $S\in \F^{[d]}$ such that $\rho(S_\ep
x', S_\ep y') <\d$ for every $\ep\neq \emptyset$.) The set of
regionally proximal pairs of order $d$ is denoted by $\RP^{[d]}$ (or
by $\RP^{[d]}(X)$ in case of ambiguity), which is called {\em the
regionally proximal relation of order $d$}.
\end{de}

It is easy to see that $\RP^{[d]}$ is a closed and invariant
relation for all $d\in \N$. Note that
\begin{equation*}
     \ldots \subseteq \RP^{[d+1]}\subseteq
    \RP^{[d]}\subseteq \ldots \subseteq \RP^{[2]}\subseteq \RP^{[1]}=\RP(X).
\end{equation*}

By the definition it is easy to verify the following equivalent
condition for $\RP^{[d]}$, see \cite{HKM}.
\begin{lem}\label{RPd-def}
Let $(X, T)$ be a minimal system and let $d\ge 1$ be an integer. Let
$x, y \in X$. Then $(x, y) \in \RP^{[d]}$ if and only if there is
some ${\bf a}_*\in X^{[d]}_*$ such that $(x, {\bf a}_*, y, {\bf
a}_*) \in \Q^{[d+1]}$.
\end{lem}

\begin{rem}\label{rem-RP1}
When $d=1$, $\RP^{[1]}$ is the classical regionally proximal
relation. If $(X,T)$ is minimal, it is easy to verify directly the
following useful fact:
$$(x,y)\in \RP=\RP^{[1]} \Leftrightarrow (x,x,y,x)\in \Q^{[2]} \Leftrightarrow (x,y,y,y)\in \Q^{[2]}.$$
\end{rem}

\section{Main results}\label{section-main}

In this section we will state the main results of the paper.

\subsection{$\F^{[d]}$-minimal sets in $\Q^{[d]}$}

To show $\RP^{[d]}$ is an equivalence relation we are forced to
investigate the $\F^{[d]}$-minimal sets in $\Q^{[d]}$ and the
equivalent conditions for $\RP^{[d]}$. Those are done in Theorem
\ref{main1} and Theorem \ref{main2} respectively.

First recall that $(\Q^{[d]}, \G^{[d]})$ is a minimal system, which
is mentioned in \cite{HKM}. But we need to know $\F^{[d]}$-minimal
sets in $\Q^{[d]}$. Let $(X, T)$ be a system and $x\in X$. Recall
that $\overline{\F^{[d]}}({\bf x})=\overline{\O({\bf x}, \F^{[d]})}$
for ${\bf x}\in X^{[d]}$.  Set
$$\Q^{[d]}[x]=\{{\bf z}\in \Q^{[d]}(X): z_\emptyset =x \}.$$

\begin{thm} \label{main1}
Let $(X, T)$ be a minimal system and $d\in \N$. Then
\begin{enumerate}
  \item $(\overline{\F^{[d]}}(x^{[d]}),\F^{[d]})$ is minimal for all
  $x\in X$.

  \item $(\overline{\F^{[d]}}(x^{[d]}), \F^{[d]})$ is the unique
$\F^{[d]}$-minimal subset in $\Q^{[d]}[x]$ for all $x\in X$.
\end{enumerate}
\end{thm}

\subsection{$\RP^{[d]}$ is an equivalence relation}

With the help of Theorem \ref{main1}, we can prove that $\RP^{[d]}$
is an equivalence relation. First we have the following equivalent
conditions for $\RP^{[d]}$.

\begin{thm}\label{main2}
Let $(X, T)$ be a minimal system and $d\in \N$. Then the following
conditions are equivalent:
\begin{enumerate}
\item $(x,y)\in \RP^{[d]}$;
\item $(x,y,y,\ldots,y)=(x,y^{[d+1]}_*)\in \Q^{[d+1]}$;
\item $(x,y,y,\ldots,y)=(x,y^{[d+1]}_*) \in
\overline{\F^{[d+1]}}(x^{[d+1]}).$
\end{enumerate}
\end{thm}

\begin{proof}
$(3)\Rightarrow (2)$ is obvious. $(2)\Rightarrow (1)$ follows from
Lemma \ref{RPd-def}. Hence it suffices to show $(1)\Rightarrow (3)$.

Let $(x,y)\in \RP^{[d]}$. Then by Lemma \ref{RPd-def} there is some
${\bf a}_*\in X^{[d]}_*$ such that $(x, {\bf a}_*, y, {\bf a}_*) \in
\Q^{[d+1]}$. Observe that $(y, {\bf a}_*)\in \Q^{[d]}$. By Theorem
\ref{main1}-(2), there is a sequence $\{F_k\}\subset \F^{[d]}$ such
that $F_k (y, {\bf a}_*)\to y^{[d]}, k\to \infty$. Hence
$$F_k\times F_k(x, {\bf a}_*, y, {\bf a}_*)\to (x,y^{[d]}_*,y,y^{[d]}_*)=
(x,y^{[d+1]}_*),\ k\to \infty.$$ Since $F_k\times F_k\in \F^{[d+1]}$
and $(x, {\bf a}_*, y, {\bf a}_*) \in \Q^{[d+1]}$, we have that
$(x,y^{[d+1]}_*)\in \Q^{[d+1]}$.

By Theorem \ref{main1}-(1), $y^{[d+1]}$ is $\F^{[d+1]}$-minimal. It
follows that $(x,y^{[d+1]}_*)$ is also $\F^{[d+1]}$-minimal. Now
$(x,y^{[d+1]}_*)\in \Q^{[d+1]}[x]$ is $\F^{[d+1]}$-minimal and by
Theorem \ref{main1}-(2), $\overline{\F^{[d+1]}}(x^{[d+1]})$ is the
unique $\F^{[d+1]}$-minimal subset in $\Q^{[d+1]}[x]$. Hence we have
that $(x,y^{[d+1]}_*) \in \overline{\F^{[d+1]}}(x^{[d+1]})$, and the
proof is completed.
\end{proof}

By Theorem \ref{main2}, we have the following theorem immediately.

\begin{thm}\label{main3}
Let $(X, T)$ be a minimal system and $d\in \N$. Then $\RP^{[d]}(X)$
is an equivalence relation.
\end{thm}

\begin{proof}
It suffices to show the transitivity, i.e. if $(x,y), (y,z)\in
\RP^{[d]}(X)$, then $(x,z)\in \RP^{[d]}(X)$. Since $(x,y), (y,z)\in
\RP^{[d]}(X)$, by Theorem \ref{main2} we have
$$(y, x,x,\ldots, x), (y, z,z,\ldots,z)\in
\overline{\F^{[d+1]}}(y^{[d+1]}).$$ By Theorem \ref{main1}
$(\overline{\F^{[d+1]}}(y^{[d+1]}),\F^{[d+1]})$ is minimal, it
follows that $(y, z,z,\ldots,z)\in
\overline{\F^{[d+1]}}(y,x,x,\ldots, x)$. It follows that $(x,
z,z,\ldots,z)\in \overline{\F^{[d+1]}}(x^{[d+1]})$. By Theorem
\ref{main2} again, $(x,z)\in \RP^{[d]}(X)$.
\end{proof}

\begin{rem}
By Theorem \ref{main2} we know that in the definition of regionally
proximal relation of $d$, $x'$ can be replaced by $x$. More
precisely, $(x,y)\in \RP^{[d]}$ if and only if for any $\d>0$ there
exist $y'\in X$ and a vector ${\bf n}= (n_1,\ldots, n_d)\in \Z^d$
such that for any nonempty $\ep\subset [d]$
$$\rho(y,y')<\d\ \text{and}\ \rho(T^{{\bf n}\cdot \ep}x,T^{{\bf n}\cdot \ep}y')<\d .$$
\end{rem}

\subsection{$\RP^{[d]}$ and nilfactors}

A subset $S\subset \Z$ is {\it thick} if it contains arbitrarily
long runs of positive integers, i.e.  there is a subsequence
$\{n_i\}$ of $\Z$ such that $S\supset \bigcup_{i=1}^\infty \{n_i,
n_i+1, \ldots, n_i+i\}$.

Let $\{ b_i \}_{i\in I}$ be a finite or infinite sequence in
$\mathbb{Z}$. One defines $$FS(\{ b_i \}_{i\in I})= \Big\{\sum_{i\in
\alpha} b_i: \alpha \text{ is a finite non-empty subset of } I\Big
\}$$ Note when $I=[d]$, $$FS(\{b_i\}_{i=1}^d)= \Big\{\sum_{i\in
I}b_i\ep_i: \ep=(\ep_i)\in \{0,1\}^d\setminus \{\emptyset\} \Big\}.
$$
$F$ is an {\it IP set} if it contains some
$FS({\{p_i\}_{i=1}^{\infty}})$, where $p_i\in\Z$.

\begin{lem}\label{prox}
Let $(X,T)$ be a system. Then for every $d\in \N$, the proximal
relation
$${\bf P}(X)\subseteq \RP^{[d]}(X).$$
\end{lem}

\begin{proof}
Let $(x,y)\in {\bf P}(X)$ and $\d>0$. Set $$N_\d(x,y)=\{n\in \Z:
\rho(T^nx,T^ny)<\d\}.$$ It is easy to check $N_\d(x,y)$ is thick and
hence an IP set. From this it follows that ${\bf P}(X)\subseteq
\RP^{[d]}(X)$. More precisely, set
$FS(\{p_i\}_{i=1}^\infty)\subseteq N_\d(x,y)$, then for any $d\in
\N$,
$$\rho(T^{p_1\ep_1+\ldots+p_d\ep_d}x,
T^{p_1\ep_1+\ldots+p_d\ep_d}y)<\d, \ \ep=(\ep_1,\ldots,\ep_d)\in
\{0,1\}^{d}, \ep\not=(0,\ldots,0).$$ That is, $(x,y)\in \RP^{[d]}$
for all $d\in \N$.
\end{proof}

The following corollary was observed in \cite{HM} for $d=2$.
\begin{cor}
Let $(X,T)$ be a minimal system and $d\in \N$. Then $(X,T)$ is a
weakly mixing system if and only if
$$\RP^{[d]}=X\times X.$$
\end{cor}

\begin{proof}
Since a minimal system $(X,T)$ is weakly mixing if and only if
$\overline{{\bf P}(X)}=\RP(X)=X\times X$ (see \cite{Au88}), so the
result follows from Lemma \ref{prox}.
\end{proof}

We remark that more interesting properties for weakly mixing systems
will be shown in Theorem \ref{wm-absolute} in the sequel.

\begin{prop}\label{diagonal}
Let $(X,T)$ be a minimal system and $d\in \N$. Then $\RP^{[d]}=\D$
if and only if $X$ is a system of order $d$.
\end{prop}

\begin{proof} It follows from Theorem \ref{main2} and Theorem
\ref{HKM} directly.
\end{proof}

\subsection{Maximal nilfactors} Note that the lifting property of $\RP^{[d]}$
between two minimal systems is obtained in the paper. This result is
new even for minimal distal systems.

\begin{thm}\label{lifting}
Let $\pi: (X,T)\rightarrow (Y,T)$ be a factor map and $d\in \N$.
Then
\begin{enumerate}
  \item $\pi\times \pi (\RP^{[d]}(X))\subseteq \RP^{[d]}(Y)$;
  \item if $(X,T)$ is minimal, then $\pi\times \pi (\RP^{[d]}(X))=
  \RP^{[d]}(Y)$.
\end{enumerate}
\end{thm}

\begin{proof} (1) It follows from the definition.

(2) It will be proved in Section \ref{section-lift}.
\end{proof}

\begin{thm}\label{nil-factor}
Let $\pi: (X,T)\rightarrow (Y,T)$ be a factor map of minimal systems
and $d\in \N$. Then the following conditions are equivalent:
\begin{enumerate}
  \item $(Y,T)$ is a $d$-step nilsystem;
  \item $\RP^{[d]}(X)\subset R_\pi$.
\end{enumerate}
Especially the quotient of $X$ under $\RP^{[d]}(X)$ is the maximal
$d$-step nilfactor of $X$, i.e. any $d$-step nilfactor of $X$ is the
factor of $X/ \RP^{[d]}(X)$.
\end{thm}


\begin{proof}
Assume that $(Y,T)$ is a $d$-step nilsystem. Then we have 
$\RP^{[d]}(Y)=\D_Y$ by Proposition \ref{diagonal}. Hence by Theorem
\ref{lifting}-(1),
$$\RP^{[d]}(X)\subset (\pi\times \pi)^{-1}(\D_Y)=R_\pi.$$

Conversely, assume that $\RP^{[d]}(X)\subset R_\pi$. If $(Y,T)$ is
not a $d$-step nilsystem, then by Proposition \ref{diagonal},
$\RP^{[d]}(Y)\neq \D_Y$. Let $(y_1, y_2)\in \RP^{[d]}\setminus
\D_Y$. Now by Theorem \ref{lifting}, there are $x_1,x_2\in X$ such
that $(x_1, x_2)\in \RP^{[d]}(X)$ with $(\pi\times
\pi)(x_1,x_2)=(y_1,y_2)$. Since $\pi(x_1)=y_1\neq y_2=\pi(x_2)$,
$(x_1,x_2)\not \in R_\pi$. This means that $\RP^{[d]}(X)\not \subset
R_\pi$, a contradiction! The proof is completed.
\end{proof}

\begin{rem}
In \cite[Proposition 4.5]{HKM} it is showed that this proposition
holds for minimal distal systems.
\end{rem}

\subsection{Weakly mixing systems} In this subsection we completely
determine $\Q^{[d]}$ and $\overline{\F^{[d]}}(x^{[d]})$ for minimal
weakly mixing systems.
\begin{thm}\label{wm-absolute}
Let $(X,T)$ be a minimal weakly mixing system and $d\ge 1$. Then
\begin{enumerate}
  \item $(\Q^{[d]}, \G^{[d]})$ is minimal and $\Q^{[d]}=X^{[d]}$;
  \item For all $x\in X$, $(\overline{\F^{[d]}}(x^{[d]}), \F^{[d]})$
  is minimal and $$\overline{\F^{[d]}}(x^{[d]})=\{x\}\times X^{[d]}_*=\{x\}\times
  X^{2^d-1}.$$
\end{enumerate}
\end{thm}

\begin{proof}
The fact that $(\Q^{[d]}, \G^{[d]})$ is minimal and
$\Q^{[d]}=X^{[d]}$ is followed from (2) easily. Hence it suffices to
show (2).

We will show for any point of ${\bf x}\in X^{[d]}$ with $x_\emptyset
=x$, we have $$\overline{\F^{[d]}}({\bf x})=\{x\}\times X^{[d]}_*,$$
which obviously implies (2). First note that it is trivial for
$d=1$. Now we assume that (1), and hence (2) hold for $d-1$, $d\ge
2$.

Let ${\bf x}=({\bf x'},{\bf x''})\in \Q^{[d]}$. Since $(X,T)$ is
weakly mixing, $(X^{[d-1]}, T^{[d-1]})$ is transitive (see
\cite{F67}). Let ${\bf a}\in X^{[d-1]}$ be a transitive point. By
the induction for $d-1$, $\Q^{[d-1]}=X^{[d-1]}$ is
$\G^{[d]}$-minimal. Hence ${\bf a}\in \overline{\O({\bf x''},
\G^{[d-1]})}$ and there is some sequence $F_k\in \F^{[d]}$ and ${\bf
w}\in X^{[d-1]}$ such that
$$F_k{\bf x}=F_k({\bf x'},{\bf x''})\to ({\bf w}, {\bf a}), \ k\to \infty. $$
Especially $({\bf w},{\bf a})\in \overline{\F^{[d]}}({\bf x})$. Note
that $$(T^{[d]}_d)^n({\bf w},{\bf a})=({\bf w},(T^{[d-1]})^n{\bf
a})\in \overline{\F^{[d]}}({\bf x}).$$ We have $$\{{\bf w}\}\times
\O({\bf a}, T^{[d-1]})\subset \overline{\F^{[d]}}({\bf x})$$ Since
${\bf a}$ is a transitive point of $(X^{[d-1]}, T^{[d-1]})$, we have
\begin{equation}\label{d1}
\{{\bf w}\}\times X^{[d-1]}=\{{\bf w}\}\times \overline{\O({\bf a},
T^{[d-1]})}\subset \overline{\F^{[d]}}({\bf x}).
\end{equation}

By the induction assumption for $d-1$, ${\bf w}$ is minimal for
$\F^{[d-1]}$ action and
\begin{equation}\label{d2}
    \overline{\F^{[d-1]}}({\bf w})=\overline{\O({\bf w},
\F^{[d-1]})}= \{x\}\times X^{[d-1]}_*.
\end{equation}
By acting the elements of $\F^{[d]}$ on (\ref{d1}), we have
\begin{equation}\label{d3}
    \O({\bf w}, \F^{[d-1]}) \times X^{[d-1]}\subset \overline{\F^{[d]}}({\bf x}).
\end{equation}
By (\ref{d2}) and (\ref{d3}), we have
$$\{x\}\times X^{[d-1]}_*\times X^{[d-1]}=\{x\}\times X^{[d]}_*\subset \overline{\F^{[d]}}({\bf x}).$$
This completes the proof.
\end{proof}

\begin{rem} We remark that using the so-called natural extension, it can be
shown that the main results of the paper hold for continuous
surjective maps.
\end{rem}

\section{$\F^{[d]}$-minimal sets in $\Q^{[d]}$}\label{section-F-minimal}

In this section we discuss $\F^{[d]}$-minimal sets in $\Q^{[d]}$ and
prove Theorem \ref{main1}-(1). First we will discuss proximal
extensions, distal extensions and weakly mixing extension one by
one. They exhibit different properties and satisfy our requests by
different reasons. After that, the proof of Theorem \ref{main1}-(1)
will be given. The proof of Theorem \ref{main1}-(2) will be given in
next section. For notions which are not mentioned before see
Appendix \ref{section-tds}.

\subsection{Idea of the proof of Theorem \ref{main1}-(1)}

Before going on let us say something about the idea in the proof of
Theorem \ref{main1}-(1). By the structure theorem \ref{structure},
for a minimal system $(X,T)$, we have the following diagram.

\[
\begin{CD}
X_\infty @>{\pi}>> X\\
      @VV{\phi}V\\
Y_\infty
\end{CD}
\]

\medskip

In this diagram $Y_\infty$ is a strictly PI system, $\phi$ is weakly
mixing and RIC, and $\pi$ is proximal.

So if we want to show that $(\overline
{\F^{[d]}}(x^{[d]}),\F^{[d]})$ is minimal for all $x\in X$, it is
sufficient to show it holds for $X_\infty$. By the definition of
$X_\infty$ and $Y_\infty$, it is sufficient to consider the
following cases: (1) proximal extensions; (2) distal or
equicontinuous extensions; (3) RIC weakly mixing extensions and (4)
the inverse limit. Since the inverse limit is easy to handle, we
need only focus on the three kinds of extensions.

\subsection{Properties about three kinds of extensions}

In this subsection we collect some properties about proximal, distal
and weakly mixing extensions, which will be used frequently in the
sequel. As in Appendix \ref{section-tds}, $(X,\t)$ is a system under
the action of a topological group $\t$, and $E(X,\t)$ is its
enveloping semigroup.

The following two lemmas are folk results, for completeness we
include proofs.

\begin{lem}\label{proximal-lemma}
Let $\pi: (X,\t)\rightarrow (Y,\t)$ be a proximal extension of
minimal systems. Let $x\in X, y=\pi(x)$ and let $x_1,x_2,\ldots,
x_n\in \pi^{-1}(y)$. Then there is some $p\in E(X,\t)$ such that
$$px_1=px_2=\ldots=px_n=x.$$
Especially, when $x=x_1$, we have that $(x_1, x_2, \ldots, x_n)$ is
proximal to $(x, x, \ldots , x)$ in $(X^n, \t)$.
\end{lem}

\begin{proof}
Since $(x_1, x_2)\in R_\pi\subset {\bf P}(X,\t)$, by Proposition
\ref{proximal-Ellis} there is some $p\in E(X,\t)$ such that
$px_1=px_2$.

Now assume that for $2\le j\le n-1$, there is some $p_1\in E(X,\t)$
such that $p_1x_1=p_1x_2=\ldots =p_1x_j$. Since $R_\pi$ is closed
and invariant and $(x_j,x_{j+1})\in R_\pi$, $(p_1x_j, p_1x_{j+1})\in
R_\pi \subset {\bf P}(X,\t)$. So by Proposition \ref{proximal-Ellis}
there is $p_2\in E(X,\t)$ such that $p_2(p_1x_j)=p_2(p_1x_{j+1})$.
Let $p=p_2p_1$, then we have
$$px_1=px_2=\ldots =px_j=px_{j+1}.$$
Inductively, there is some $p\in E(X,\t)$ such that
$$px_1=px_2=\ldots=px_n .$$
Since $(X,\t)$ is minimal, we can assume that they are equal to $x$.

If $x_1=x$, then $px_1=px_2=\ldots=px_n=x=x_1$ and hence
$$p(x_1, x_2, \ldots, x_n)=(x, x, \ldots,x)=p(x, x, \ldots, x).$$
That is, $(x_1, x_2, \ldots, x_n)$ is proximal to $(x, x, \ldots ,
x)$ in $(X^n, \t)$.
\end{proof}

\begin{lem}\label{distal-lemma1}
Let $\pi: (X,\t)\rightarrow (Y,\t)$ be a distal extension of
systems. Then for any $x\in X$, if $\pi(x)$ is minimal in $(Y,\t)$,
then $x$ is minimal in $(X,\t)$. Especially, if $(Y,\t)$ is
semi-simple (i.e. every point is minimal), then so is $(X,\t)$.
\end{lem}

\begin{proof}
Let $x\in X$ and $y=\pi(x)$. Since $y$ is a minimal point, by
Proposition \ref{minimal-point} there is some minimal idempotent
$u\in E(X,\t)$ such that $uy=y$. Then $\pi(ux)=u\pi(x)=uy=y$. Hence
$ux,x\in \pi^{-1}(y)$. Since $(ux,x)\in {\bf P}(X,\t )$ (Proposition
\ref{proximal-Ellis}) and $\pi$ is distal, we have $ux=x$. That is,
$x$ is a minimal point of $X$ by Proposition \ref{minimal-point}.
\end{proof}

Now we discuss weakly mixing extensions. We need Theorem
\ref{wm-extension}, which is a generalization of \cite[Chapter 14,
Theorem 28]{Au88}. Note that in \cite[Theorem 2.7 and Corollary
2.9]{Gl05} Glasner showed that $R_\pi^n$ is transitive. So Theorem
\ref{wm-extension} is a slight strengthening of the results in
\cite{Gl05}. Since its proof needs some techniques in the enveloping
semigroup theory, we leave it to the appendix.

\begin{thm}\label{wm-extension}
Let $\pi : (X, \t)\rightarrow (Y, \t)$ be a RIC weakly mixing
extension of minimal systems, then for all $n\ge 1$ and $y\in Y$,
there exists a transitive point $(x_1,x_2,\ldots,x_n)$ of $R^n_\pi$
with $x_1, x_2, \ldots , x_n\in \pi^{-1}(y)$.
\end{thm}

Note that each RIC extension is open, and $\pi: X\rightarrow Y$ is
open if and only if $Y\rightarrow 2^X, y\mapsto \pi^{-1}(y)$ is
continuous, see for instance \cite{Vr}. Using Theorem
\ref{wm-extension} we have the following lemma, which will be used
in the sequel.

\begin{lem}\label{wm-lemma}
Let $\pi :(X,T)\rightarrow (Y,T)$ be a RIC weakly mixing extension
of minimal systems of $(X,T)$ and $(Y,T)$. Then for each $y\in Y$
and $d\ge 1$, we have
\begin{enumerate}
  \item
  $\left(\pi^{-1}(y)\right)^{[d]}=\left(\pi^{-1}(y)\right)^{2^d}\subset
  \Q^{[d]}(X)$,
  \item for all ${\bf x}\in X^{[d]}$ with $x_\emptyset =x$ and
$\pi^{[d]}({\bf x})=y^{[d]}$
$$\{x\}\times
\left(\pi^{-1}(y)\right)^{[d]}_*=\{x\}\times
\left(\pi^{-1}(y)\right)^{2^d-1}\subset \overline{\F^{[d]}}({\bf
x}).$$
\end{enumerate}
\end{lem}

\begin{proof}
The idea of proof is similar to Theorem \ref{wm-absolute}. When
$d=1$, for any $(x,x')\in X^{[1]}=X\times X$,
$\overline{\F^{[1]}}(x,x')=\overline{\O}\left((x,x'), \id\times
T\right)=\{x\}\times X$ and $\Q^{[1]}(X)=X\times X$. Hence the
results hold obviously. Now we show the case for $d=2$. Let ${\bf x
}=(x_1,x_2,x_3,x_4)\in X^{[2]}$ with
$\pi^{[2]}(x_1,x_2,x_3,x_4)=y^{[2]}$. By Theorem \ref{wm-extension},
there is a transitive point $(a,b)$ of $(R_\pi, T\times T)$ with
$\pi(a)=\pi(b)=y$. Since $(X,T)$ is minimal, there is some sequence
$\{n_i\}\subset \Z$ such that $T^{n_i}x_3\to a, i\to \infty$.
Without loss of generality, assume that $T^{n_i}x_4\to x_4', i\to
\infty$ for some $x_4'\in X$. Since $\pi(a)=y$, $\pi(x_4')=y$ too.
So
\begin{equation}\label{a1}
(\id \times \id \times T\times T)^{n_i}(x_1,x_2,x_3,x_4)\to
(x_1,x_2,a,x_4'), \ i\to \infty.
\end{equation}
Since $(X,T)$ is minimal, there is some sequence $\{m_i\}\subset \Z$
such that $T^{m_i}x_4'\to b, i\to \infty$. Without loss of
generality, assume that $T^{m_i}x_2\to x_2', i\to \infty$ for some
$x_2'\in X$. Since $\pi(b)=y$, $\pi(x_2')=y$ too. So
\begin{equation}\label{a2}
(\id \times T \times \id \times T)^{m_i}(x_1,x_2,a,x_4')\to
(x_1,x_2',a, b), \ i\to \infty.
\end{equation}
Hence by (\ref{a1}) and (\ref{a2}),
\begin{equation}\label{}
    (x_1,x_2',a,b)\in \overline{\F^{[2]}}({\bf x}).
\end{equation}
Thus for all $n\in \Z$,
$$(x_1,x_2',T^n a,T^nb)=(\id \times \id \times T\times T)^n(x_1,x_2',a,b)
\in \overline{\F^{[2]}}({\bf x}).$$ Since $(a,b)$ is a transitive
point of $(R_\pi, T\times T)$, it follows that
\begin{equation}\label{}
    \{x_1\}\times \{x_2'\}\times \pi^{-1}(y)\times
    \pi^{-1}(y)\subset \{x_1\}\times \{x_2'\}\times R_\pi \subset
\overline{\F^{[2]}}({\bf x}).
\end{equation}
Now we show that
\begin{equation}\label{a3}
 \{x_1\}\times \pi^{-1}(y)\times \pi^{-1}(y)\times
\pi^{-1}(y)=\{x_1\}\times ( \pi^{-1}(y))^3\subset
\overline{\F^{[2]}}({\bf x}).
\end{equation}
For any $z\in \pi^{-1}(y)$, there is a sequence $k_i\subset \Z$ such
that $T^{k_i}x_2'\to z, i\to \infty$. Thus
$T^{k_i}y=T^{k_i}\pi(x_2')=\pi(T^{k_i}x_2')\to \pi(z)=y, i\to
\infty$. Since $\pi$ is open, we have
$T^{k_i}\pi^{-1}(y)=\pi^{-1}(T^{k_i}y)\to \pi^{-1}(y), i\to \infty$
in the Hausdorff metric. Thus
\begin{equation*}
\{x_1\}\times \{z\}\times \pi^{-1}(y)^2\subset \overline{
\cup_{i=1}^\infty(\id \times T\times\id\times T)^{k_i}(\{x_1\}\times
\{x_2'\}\times \pi^{-1}(y)^2)} \subset\overline{\F^{[2]}}({\bf x}).
\end{equation*}
Since $z$ is arbitrary, we have (\ref{a3}). Similarly, we have
$\left(\pi^{-1}(y)\right)^4\subset \Q^{[2]}(X)$ and we are done for
$d=2$.

\medskip

Now assume we have (1) and (2) for $d-1$ already, and show the case
for $d$. Let ${\bf x}\in X^{[d]}$ with $x_\emptyset =x$ and
$\pi^{[d]}({\bf x})=y^{[d]}$.

Let ${\bf x}=({\bf x'},{\bf x''})$. Since $\pi$ is weakly mixing,
$(R^{2^{d-1}}_\pi, T^{[d-1]})$ is transitive. By Theorem
\ref{wm-extension} there is ${\bf a}\in R^{2^{d-1}}_\pi$ which is a
transitive point of $(R^{2^{d-1}}_\pi, T^{[d-1]})$ and
$\pi^{[d-1]}({\bf a})=y^{[d-1]}$. Without loss of generality, we may
assume that $a_\emptyset =x''_\emptyset$ (i.e. the first coordinate
of ${\bf a}$ is equal to that of ${\bf x''}$), otherwise we may use
the face transformation $\id^{[d-1]}\times T^{[d-1]}$ to find some
point in $\overline{\F^{[d]}}({\bf x})$ satisfying this property.

By the induction assumption for $d-1$,
$${\bf a}\in \{x_\emptyset''\}\times \left(\pi^{-1}(y)\right)^{2^{d-1}-1}\subset
  \overline{\F^{[d-1]}}({\bf
x''}).$$ Hence there is some sequence $F_k\in \F^{[d-1]}$ and ${\bf
w}\in X^{[d-1]}$ such that
$$F_k\times F_k ({\bf x})=F_k\times F_k({\bf x'},{\bf x''})\to ({\bf w}, {\bf a}), \ k\to \infty. $$
Especially $({\bf w},{\bf a})\in \overline{\F^{[d]}}({\bf x})$.
Since $\pi^{[d]}({\bf x})=y^{[d]}$ and $\pi^{[d-1]}({\bf
a})=y^{[d-1]}$, it is easy to verify that $\pi^{[d-1]}({\bf
w})=y^{[d-1]}$ and $w_\emptyset =x$. Note that
$$(T^{[d]}_d)^n({\bf w},{\bf a})=({\bf w},(T^{[d-1]})^n{\bf a})\in
\overline{\F^{[d]}}({\bf x}).$$ We have
$$\{{\bf w}\}\times \O({\bf a}, T^{[d-1]})\subset
\overline{\F^{[d]}}({\bf x}).$$ And so
\begin{equation}\label{d4}
\{{\bf w}\}\times \left(\pi^{-1}(y)\right)^{2^{d-1}}\subset \{{\bf
w}\}\times R^{2^{d-1}}_\pi=\{{\bf w}\}\times \overline{\O({\bf a},
T^{[d-1]})}\subset \overline{\F^{[d]}}({\bf x}).
\end{equation}
By the induction assumption for $d-1$, for ${\bf w}$ we have
\begin{equation}\label{d5}
    \{x\}\times \left(\pi^{-1}(y)\right)^{2^{d-1}-1}\subset \overline{\F^{[d-1]}}(\bf w).
\end{equation}
Hence for all ${\bf z}\in  \{x\}\times
\left(\pi^{-1}(y)\right)^{2^{d-1}-1}$, there is some sequence
$\{H_k\}\subset \F^{[d-1]}$ such that $H_k{\bf w}\to {\bf z}, k\to
\infty$. Since $\pi$ is open, similar to the proof of (\ref{a3}), we
have that $H_k\left(\pi^{-1}(y)\right)^{2^{d-1}}\to
\left(\pi^{-1}(y)\right)^{2^{d-1}}, k\to \infty$. Hence
$$H_k\times H_k\ \big(\{{\bf w}\}\times \left(\pi^{-1}(y)\right)^{2^{d-1}}\big )\to
\{{\bf z}\}\times \left(\pi^{-1}(y)\right)^{2^{d-1}}, k\to \infty.$$

Since $H_k\times H_k \in \F^{[d]}$ and ${\bf z}\in  \{x\}\times
\left(\pi^{-1}(y)\right)^{2^{d-1}-1}$ is arbitrary, it follows from
(\ref{d4}) that
$$\{x\}\times \left(\pi^{-1}(y)\right)^{2^{d-1}-1}\times
\left(\pi^{-1}(y)\right)^{2^{d-1}}=\{x\}\times
\left(\pi^{-1}(y)\right)^{2^d-1}
 \subset \overline{\F^{[d]}}({\bf x}).$$
Now by this fact it is easy to get
$\left(\pi^{-1}(y)\right)^{[d]}=\left(\pi^{-1}(y)\right)^{2^d}\subset
\Q^{[d]}(X)$. So (1) and (2) hold for the case $d$. This completes
the proof.
\end{proof}

In fact with a small modification of the above proof one can show
that $R^{2^d}_\pi\subset \Q^{[d]}(X)$. We do not know if
$\{x\}\times R^{2^d-1}_\pi\subset\overline{\F^{[d]}}({\bf x}).$

\subsection{Proof of Theorem \ref{main1}-(1)}

A subset $S\subseteq \Z$ is a {\em central set} if there exists a
system $(X,T)$, a point $x\in X$ and a minimal point $y\in X$
proximal to $x$, and a neighborhood $U_y$ of $y$ such that
$N(x,U_y)\subset S$, where $N(x,U_y)=\{n\in\Z:T^nx\in U_y\}$. It is
known that any central set is an IP-set \cite[Proposition 8.10.]{F}.

\begin{prop}\label{proximal-extension1}
Let $\pi: (X,T)\rightarrow (Y,T)$ be a proximal extension of minimal
systems and $d\in \N$. If $(\overline {\F^{[d]}}(y^{[d]}),\F^{[d]})$
is minimal for all $y\in Y$, then $(\overline
{\F^{[d]}}(x^{[d]}),\F^{[d]})$ is minimal for all $x\in X$.
\end{prop}

\begin{proof}
It is sufficient to show that for any ${\bf x}\in \overline
{\F^{[d]}}(x^{[d]})$, we have $x^{[d]}\in \overline {\F^{[d]}}({\bf
x})$. Let $y=\pi(x)$. Then by the assumption $(\overline
{\F^{[d]}}(y^{[d]}), \F^{[d]})$ is minimal. Note that $\pi^{[d]}:
(\overline {\F^{[d]}}(x^{[d]}),\F^{[d]}) \rightarrow (\overline
{\F^{[d]}}(y^{[d]}),\F^{[d]})$ is a factor map. Especially there is
some ${\bf y}\in \overline {\F^{[d]}}(y^{[d]})$ such that
$\pi^{[d]}({\bf x})={\bf y}$.

Since ${\bf y}\in \overline {\F^{[d]}}(y^{[d]})$ and $(\overline
{\F^{[d]}}(y^{[d]}), \F^{[d]})$ is minimal, there is some sequence
$F_k\in \F^{[d]}$ such that $$F_k {\bf y}\to y^{[d]},\ k\to \infty
.$$ Without loss of generality, we may assume that
\begin{equation}\label{d8}
    F_k {\bf x}\to {\bf z}, \ k\to \infty.
\end{equation}
Then $\pi^{[d]}({\bf z})=\lim_k \pi^{[d]}(F_k {\bf x})=\lim_k F_k
{\bf y}=y^{[d]}$. That is,
$$z_\ep \in \pi^{-1}(y), \ \forall \ep\in \{0,1\}^d.$$
Since $\pi$ is proximal, by Lemma \ref{proximal-lemma} there is some
$p\in E(X,T)$ such that $$pz_\ep=px=x, \quad \forall \ep\in
\{0,1\}^d.$$ That is, $p{\bf z}=x^{[d]}=px^{[d]}$, i.e. ${\bf z}$ is
proximal to $x^{[d]}$ under the action of $T^{[d]}$. Since $x^{[d]}$
is $T^{[d]}$-minimal, for any neighborhood ${\bf U}$ of ${x^{[d]}}$,
$$N_{T^{[d]}}({\bf z}, {\bf U})=\{n\in \Z: (T^{[d]})^n {\bf z}\in {\bf U}\}$$
is a central set and hence contains some IP set
$FS(\{p_i\}_{i=1}^\infty)$. Particularly,
$$FS(\{p_i\}_{i=1}^d)\subseteq N_{T^{[d]}}({\bf z}, {\bf U}).$$
This means for all $\ep\in \{0,1\}^d\setminus \{{\bf 0}\}$,
$$(T^{[d]})^{{\bf p}\cdot \ep}{\bf z}\in {\bf U},$$ where ${\bf p}=(p_1,p_2,\ldots, p_d)\in
\Z^d$. Especially, $$(T^{{\bf p}\cdot \ep}z_\ep)_{\ep\in
\{0,1\}^d}\in {\bf U}$$ In other words, we have
$$(T^{[d]}_1)^{p_1}(T^{[d]}_2)^{p_2}\ldots (T^{[d]}_d)^{p_d}{\bf z}\in {\bf U}.$$
Since ${\bf U}$ is arbitrary, we have that $x^{[d]}\in
\overline{\F^{[d]}}({\bf z})$. Combining with (\ref{d8}), we have
$$x^{[d]}\in \overline{\F^{[d]}}({\bf x}).$$
Thus $(\overline {\F^{[d]}}(x^{[d]}),\F^{[d]})$ is minimal. This
completes the proof.
\end{proof}

\begin{prop}\label{distal-extension1}
Let $\pi: (X,T)\rightarrow (Y,T)$ be a distal extension of minimal
systems and $d\in \N$. If $(\overline {\F^{[d]}}(y^{[d]}),\F^{[d]})$
is minimal for all $y\in Y$, then $(\overline
{\F^{[d]}}(x^{[d]}),\F^{[d]})$ is minimal for all $x\in X$.
\end{prop}

\begin{proof}
It follows from Lemma \ref{distal-lemma1}, since it is easy to check
that $\pi^{[d]}: (\overline {\F^{[d]}}(x^{[d]}),\F^{[d]})
\rightarrow (\overline {\F^{[d]}}(y^{[d]}),\F^{[d]})$ is a distal
extension.
\end{proof}

\begin{prop}\label{wm-extension1}
Let $\pi : (X, T)\rightarrow (Y, T)$ be a RIC weakly mixing
extension of minimal systems and $d\in \N$. If $(\overline
{\F^{[d]}}(y^{[d]}),\F^{[d]})$ is minimal for all $y\in Y$, then
$(\overline {\F^{[d]}}(x^{[d]}),\F^{[d]})$ is minimal for all $x\in
X$.
\end{prop}

\begin{proof}
It is sufficient to show that for any ${\bf x}\in \overline
{\F^{[d]}}(x^{[d]})$, we have $x^{[d]}\in \overline {\F^{[d]}}({\bf
x})$. Let $y=\pi(x)$. Then by the assumption $(\overline
{\F^{[d]}}(y^{[d]}), \F^{[d]})$ is minimal. Note that $\pi^{[d]}:
(\overline {\F^{[d]}}(x^{[d]}),\F^{[d]}) \rightarrow (\overline
{\F^{[d]}}(y^{[d]}),\F^{[d]})$ is a factor map. Especially there is
some ${\bf y}\in \overline {\F^{[d]}}(y^{[d]})$ such that
$\pi^{[d]}({\bf x})={\bf y}$.

Since ${\bf y}\in \overline {\F^{[d]}}(y^{[d]})$ and $(\overline
{\F^{[d]}}(y^{[d]}), \F^{[d]})$ is minimal, there is some sequence
$F_k\in \F^{[d]}$ such that $$F_k {\bf y}\to y^{[d]},\ k\to \infty
.$$ Without loss of generality, we may assume that
\begin{equation}\label{d7}
    F_k {\bf x}\to {\bf z}, \ k\to \infty.
\end{equation}
Then $\pi^{[d]}({\bf z})=\lim_k \pi^{[d]}(F_k {\bf x})=\lim_k F_k
{\bf y}=y^{[d]}$. By Lemma \ref{wm-lemma}
$$x^{[d]}\in \{x\}\times
\left(\pi^{-1}(y)\right)^{2^d-1}\subset \overline{\F^{[d]}}({\bf
z}).$$ Together with (\ref{d7}), we have $x^{[d]}\in \overline
{\F^{[d]}}({\bf x})$. This completes the proof.
\end{proof}

\medskip

\noindent {\bf Proof of Theorem \ref{main1}-(1):} By the structure
theorem \ref{structure}, we have the following diagram, where
$Y_\infty$ is a strictly PI-system, $\phi$ is RIC weakly mixing
extension and $\pi$ is proximal.

\[
\begin{CD}
X_\infty @>{\pi}>> X\\
      @VV{\phi}V\\
Y_\infty
\end{CD}
\]

\medskip

Since the inverse limit of minimal systems is minimal, it follows
from Propositions \ref{proximal-extension1}, \ref{distal-extension1}
that the result holds for $Y_\infty$. By Proposition
\ref{wm-extension1} it also holds for $X_\infty$. Since the factor
of a minimal system is always minimal, it is easy to see that we
have the theorem for $X$. \hfill $\square$

\subsection{Minimality of $(\Q^{[d]}, \G^{[d]})$} We will need the
following theorem mentioned in \cite{HKM}, where no proof is
included. We give a proof (due to Glasner-Ellis) here for
completeness. Note one can also prove this result using the method
in the previous subsection.
\begin{prop}\label{minimal}
Let $(X, T)$ be a minimal system and let $d \ge 1$ be an integer.
Let $A$ be a $T^{[d]}$-minimal subset of $X^{[d]}$ and set
$N=\overline{\O(A, \F^{[d]})}=\cl\big( \bigcup\{SA: S\in
\F^{[d]}\}\big )$. Then $(N, \G^{[d]})$ is a minimal system, and
$\F^{[d]}$-minimal points are dense in $N$.
\end{prop}

\begin{proof}
The proof is similar to the one in \cite{G96}. Let $E=E(N,\G^{[d]})$
be the enveloping semigroup of $(N, \G^{[d]})$. Let $\pi_\ep:
N\rightarrow X$ be the projection of $N$ on the $\ep$-th component,
$\ep\in \{0,1\}^d$. We consider the action of the group $\G^{[d]}$
on the $\ep$-th component via the representation $T^{[d]}\mapsto T$
and
$$T^{[d]}_j\mapsto \left\{
                     \begin{array}{ll}
                       T, & \hbox{$j\in \ep$;} \\
                       {\rm id}, & \hbox{$j\not \in \ep$.}
                     \end{array}
                   \right.
$$
With respect to this action of $\G^{[d]}$ on $X$ the map $\pi_\ep$
is a factor map $\pi_\ep: (N,\G^{[d]})\rightarrow (X,\G^{[d]})$. Let
$\pi^*_\ep: E(N,\G^{[d]})\rightarrow E(X, \G^{[d]})$ be the
corresponding homomorphism of enveloping semigroups. Notice that for
this action of $\G^{[d]}$ on $X$ clearly $E(X,\G^{[d]})=E(X,T)$ as
subsets of $X^X$.

Let now $u\in E(N,T^{[d]})$ be any minimal idempotent in the
enveloping semigroup of $(N,T^{[d]})$. Choose $v$ a minimal
idempotent in the closed left ideal $E(N,\G^{[d]})u$. Then $vu=v$,
i.e. $v<_L u $. Set for each $\ep\in \{0,1\}^d$, $u_\ep=\pi^*_\ep u$
and $v_\ep=\pi^*_\ep v$. We want to show that also $uv=u$, i.e.
$u<_L v$. Note that as an element of $E(N, \G^{[d]})$ is determined
by its projections, it suffices to show that for each $\ep\in
\{0,1\}^d$, $u_\ep v_\ep=u_\ep$.

Since for each $\ep\in \{0,1\}^d$ the map $\pi^*_\ep$ is a semigroup
homomorphism, we have $v_\ep u_\ep=v_\ep$ as $vu=v$. In particular
we deduce that $v_\ep$ is an element of the minimal left ideal of
$E(X,T)$ which contains $u_\ep$. In turn this implies
$$u_\ep v_\ep= u_\ep v_\ep u_\ep =u_\ep ;$$
and it follows that indeed $uv=u$. Thus $u$ is an element of the
minimal left ideal of $E(N,\G^{[d]})$ which contains $v$, and
therefore $u$ is a minimal idempotent of $E(N,\G^{[d]})$.

Now let $x$ be an arbitrary point in $A$ and let $u \in E(N,
T^{[d]})$ be a minimal idempotent with $ux=x$. By the above
argument, $u$ is also a minimal idempotent of $E(N,\G^{[d]})$,
whence $N=\overline{\O(A, \F^{[d]})}=\overline{\O(x,\G^{[d]})}$ is
$\G^{[d]}$-minimal.

Finally, we show $\F^{[d]}$-minimal points are dense in $N$. Let
$B\subseteq N$ be an $\F^{[d]}$-minimal subset. Then $\O(B,
T^{[d]})=\bigcup \{(T^{[d]})^nB: n\in \Z\}$ is a
$\G^{[d]}$-invariant subset of $N$. Since $(N,\G^{[d]})$ is minimal,
$\O(B, T^{[d]})$ is dense in $N$. Note that every point in $\O(B,
T^{[d]})$ is $\F^{[d]}$-minimal, hence the proof is completed.
\end{proof}

Setting $A=\Delta^{[d]}$ we have
\begin{cor}\label{Q-is-minimal}
Let $(X, T)$ be a minimal system and let $d \ge 1$ be an integer.
Then $(\Q^{[d]}, \G^{[d]})$ is a minimal system, and
$\F^{[d]}$-minimal points are dense in $\Q^{[d]}$.
\end{cor}

\section{Proof of Theorem \ref{main1}-(2)}\label{section-RP-equi}

In this section we prove Theorem \ref{main1}-(2). That is, we show
that $(\overline{\F^{[d]}}(x^{[d]}), \F^{[d]})$ is the unique
$\F^{[d]}$-minimal subset in $\Q^{[d]}[x]$ for all $x\in X$.

\subsection{A useful lemma}
The following lemma is a key step to show the uniqueness of minimal
sets in $\Q^{[d]}[x]$ for $x\in X$. Unlike the case when $(X,T)$ is
minimal distal, we need to use the enveloping semigroup theory.
\begin{lem}\label{host-trick}
Let $(X, T)$ be a minimal system and let $d \ge 1$ be an integer. If
$(x^{[d-1]}, {\bf w})\in \Q^{[d]}(X)$ for some ${\bf w}\in
X^{[d-1]}$ and it is $\F^{[d]}$-minimal, then
$$(x^{[d-1]}, {\bf w})\in \overline{\F^{[d]}}(x^{[d]}).$$
\end{lem}

\begin{proof}
Since $(x^{[d-1]}, {\bf w})\in \Q^{[d]}(X)$ and $(\Q^{[d]},
\G^{[d]})$ is a minimal system by Corollary \ref{Q-is-minimal},
$(x^{[d-1]},{\bf w})$ is in the $\G^{[d]}$-orbit closure of
$x^{[d]}$, i.e. there are sequences $\{n_k\}_k, \{n_k^1\}_k, \ldots,
\{n_k^d\}_k\subseteq \Z$ such that
$$(T^{[d]}_d)^{n_k}(T^{[d]}_1)^{n^1_k}\ldots
(T^{[d]}_{d-1})^{n^{d-1}_k}
(T^{[d]})^{n^d_k}(x^{[d-1]},x^{[d-1]})\to (x^{[d-1]}, {\bf w}), \
k\to \infty.$$ Let $${\bf a_k}= (T^{[d-1]}_1)^{n^1_k}\ldots
(T^{[d-1]}_{d-1})^{n^{d-1}_k} (T^{[d-1]})^{n^d_k}(x^{[d-1]}),$$ then
the above limit can be rewritten as
\begin{equation}\label{q1}
(T^{[d]}_d)^{n_k}({\bf a_k},{\bf a_k})=({\rm id}^{[d-1]}\times
T^{[d-1]})^{n_k}({\bf a_k},{\bf a_k})\to (x^{[d-1]}, {\bf w}), \
k\to \infty.
\end{equation}

Let
\begin{equation*}
    \pi_1: (X^{[d]}, \F^{[d]})\rightarrow (X^{[d-1]},
    \F^{[d]}), \quad ({\bf x'},{\bf x''})\mapsto {\bf x'},
\end{equation*}
\begin{equation*}
    \pi_2: (X^{[d]}, \F^{[d]})\rightarrow (X^{[d-1]},
    \F^{[d]}), \quad ({\bf x'},{\bf x''})\mapsto {\bf x''},
\end{equation*}
be projections to the first $2^{d-1}$ coordinates and last $2^{d-1}$
coordinates respectively. For $\pi_1$ we consider the action of the
group $\F^{[d]}$ on $X^{[d-1]}$ via the representation
$T_i^{[d]}\mapsto T^{[d-1]}_i$ for $1\le i\le d-1$ and
$T^{[d]}_d\mapsto \id^{[d-1]}$. For $\pi_2$ we consider the action
of the group $\F^{[d]}$ on $X^{[d-1]}$ via the representation
$T_i^{[d]}\mapsto T^{[d-1]}_i$ for $1\le i\le d-1$ and
$T^{[d]}_d\mapsto T^{[d-1]}$.

Denote the corresponding semigroup homomorphisms of enveloping
semigroups by
\begin{equation*}
    \pi_1^*: E(X^{[d]}, \F^{[d]})\rightarrow E(X^{[d-1]},
    \F^{[d]}), \quad\pi_2^*: E(X^{[d]}, \F^{[d]})\rightarrow E(X^{[d-1]},
    \F^{[d]}).
\end{equation*}
Notice that for this action of $\F^{[d]}$ on $X^{[d-1]}$ clearly
$$\pi^*_1(E(X^{[d]},\F^{[d]}))=E(X^{[d-1]},\F^{[d-1]})\ \text{and}\ \pi^*_2(E(X^{[d]},\F^{[d]}))=E(X^{[d-1]},\G^{[d-1]})$$ as subsets of
$(X^{[d-1]})^{X^{[d-1]}}$. Thus for any $p\in E(X^{[d]},\F^{[d]})$
and ${\bf x}\in X^{[d]}$, we have
$$p  {\bf x}= p({\bf x'},{\bf x''})=(\pi^*_1(p){\bf x'}, \pi^*_2(p) {\bf x''}).$$

Now fix a minimal left ideal ${\bf L}$ of $E(X^{[d]}, \F^{[d]})$. By
(\ref{q1}), ${\bf a_k}\to x^{[d-1]}, k\to \infty$. Since
$(\Q^{[d-1]}(X), \G^{[d-1]})$ is minimal, there exists $p_k\in {\bf
L}$ such that ${\bf a_k}=\pi^*_2(p_k) x^{[d-1]}$. Without loss of
generality, we assume that $p_k\to p\in {\bf L}$. Then
$$\pi^*_2(p_k) x^{[d-1]}={\bf a_k}\to x^{[d-1]}\ \text{and}\
\pi^*_2(p_k)x^{[d-1]}\to \pi^*_2(p) x^{[d-1]}.$$ Hence
\begin{equation}\label{q2}
    \pi^*_2(p) x^{[d-1]}=x^{[d-1]}.
\end{equation}

Since ${\bf L}$ is a minimal left ideal and $p\in {\bf L}$, by
Proposition \ref{left-ideal} there exists a minimal idempotent $v\in
J({\bf L})$ such that $v p=p$. Then we have
$$\pi^*_2(v)x^{[d-1]}=\pi^*_2(v)\pi^*_2(p)x^{[d-1]}=\pi^*_2(vp)x^{[d-1]}
=\pi^*_2(p)x^{[d-1]}=x^{[d-1]}.$$ Let
$$F=\mathfrak{G}(\overline{\F^{[d-1]}}(x^{[d-1]}), x^{[d-1]})=\{\a\in
v{\bf L}: \pi^*_2(\a) x^{[d-1]}=x^{[d-1]}\}$$ be the Ellis group.
Then $ F$ is a subgroup of the group $v{\bf L}$. By (\ref{q2}), we
have that $p\in F$.

Since $F$ is a group and $p\in F$. We have
\begin{equation}\label{q3}
    pFx^{[d]}=Fx^{[d]}\subset \pi_2^{-1}(x^{[d-1]}).
\end{equation}
Since $vx^{[d]}\in Fx^{[d]}$, there is some ${\bf x_0}\in F x^{[d]}$
such that $v x^{[d]}=p {\bf x_0}$. Set ${\bf x_k}=p_k {\bf x_0}$.
Then
$${\bf x_k}=p_k {\bf x_0}\to p { \bf x_0}=v x^{[d]}=(\pi^*_1(v)x^{[d-1]},
x^{[d-1]}),\ k\to \infty,$$ and $$\pi_2({\bf x_k})=\pi_2(p_k {\bf
x_0})=\pi^*_2(p_k) x^{[d-1]}= {\bf a_k} \to x^{[d-1]}, \ k\to
\infty.$$ Let ${\bf x_k}=({\bf b_k}, {\bf a_k})\in
\overline{\F^{[d]}}(x^{[d]})$. Then $\lim_k {\bf
b_k}=\pi^*_1(v)x^{[d-1]}$.

By (\ref{q1}), we have $(T^{[d-1]})^{n_k}{\bf a_k}\to {\bf w}, k\to
\infty$. Hence
\begin{equation}\label{}
    ({\rm id}^{[d-1]}\times T^{[d-1]})^{n_k}({\bf b_k},{\bf a_k})=
    ({\bf b_k}, (T^{[d-1]})^{n_k}{\bf a_k})\to
    (\pi^*_1(v)x^{[d-1]}, {\bf w}), \ k\to \infty.
\end{equation}
Since ${\rm id}^{[d-1]}\times T^{[d-1]}=T^{[d]}_d\in \F^{[d]}$ and
$({\bf b_k}, {\bf a_k})\in \overline{\F^{[d]}}(x^{[d]})$, we have
\begin{equation}\label{q4}
(\pi^*_1(v)x^{[d-1]}, {\bf w})\in \overline{\F^{[d]}}(x^{[d]}).
\end{equation}
Since $(x^{[d-1]},{\bf w})$ is $\F^{[d]}$ minimal by assumption, by
Proposition \ref{minimal-point} there is some minimal idempotent
$u\in J({\bf L})$ such that
$$u(x^{[d-1]},{\bf w})=(\pi^*_1(u)x^{[d-1]},\pi^*_2(u){\bf w})=(x^{[d-1]},{\bf w}).$$
Since $u,v\in {\bf L}$ are minimal idempotents in the same minimal
left ideal ${\bf L}$, we have $uv=u$ by Proposition
\ref{left-ideal}. Thus
\begin{eqnarray*}
u(\pi^*_1(v)x^{[d-1]},{\bf
w})&=&(\pi^*_1(u)\pi^*_1(v)x^{[d-1]},\pi^*_2(u){\bf w})\\
&=&(\pi^*_1(uv)x^{[d-1]},{\bf w})=(\pi^*_1(u)x^{[d-1]},{\bf
w})=(x^{[d-1]},{\bf w}).
\end{eqnarray*}
By (\ref{q4}), we have
$$(x^{[d-1]}, {\bf w})\in \overline{\F^{[d]}}(x^{[d]}).$$
The proof is completed.
\end{proof}

\subsection{Proof of Theorem \ref{main1}-(2)}

Let $(X, T)$ be a system and $x\in X$. Recall
$$\Q^{[d]}[x]=\{{\bf z}\in \Q^{[d]}(X): z_\emptyset =x \}.$$
With the help of Lemma \ref{host-trick} we have

\begin{prop}\label{unique-minimal-set}
Let $(X, T)$ be a minimal system and let $d \ge 1$ be an integer. If
${\bf x}\in \Q^{[d]}[x]$, then
$$x^{[d]}\in \overline{\F^{[d]}}({\bf x}).$$
Especially, $(\overline{\F^{[d]}}(x^{[d]}), \F^{[d]})$ is the unique
$\F^{[d]}$-minimal subset in $\Q^{[d]}[x]$.
\end{prop}

\begin{proof}
It is sufficient to show the following claim:

\medskip

\noindent {\bf S(d):} {\em If ${\bf x}\in \Q^{[d]}[x]$, then there
exists a sequence $F_k\in \F^{[d]}$ such that $F_k ( {\bf x})\to
x^{[d]}$. }

\medskip

The case {\bf S(1)} is trivial. To make the idea clearer, we show
the case when $d=2$. Let $(x,a,b,c)\in \Q^{[2]}(X)$. We may assume
that $(x,a,b,c)$ is $\F^{[2]}$-minimal, or we replace it by some
$\F^{[2]}$-minimal point in its $\F^{[2]}$ orbit closure. Since
$(X,T)$ is minimal, there is a sequence $\{n_k\}\subset \Z$ such
that $T^{n_k}a\to x$. Without loss of generality we assume
$T^{n_k}c\to c'$. Then we have
$$(T^{[2]}_1)^{n_k}(x,a,b,c)=(\id\times T\times \id
\times T)^{n_k}(x,a,b,c)\to (x,x,b,c'), \ k\to \infty.$$ Since
$(x,a,b,c)$ is $\F^{[2]}$-minimal, $(x,x,b,c')$ is also
$\F^{[2]}$-minimal. By Lemma \ref{host-trick}, $(x,x,b,c')\in
\overline{\F^{[2]}}(x^{[2]})$. Together with ${\rm id}\times T\times
{\rm id}\times T  = T^{[2]}_1 \in \F^{[2]}$ and the minimality of
the system $(\overline{\F^{[2]}}(x^{[2]}),\F^{[2]})$ (Theorem
\ref{main1}-(1)), it is easy to see there exists a sequence $F_k\in
\F^{[2]}$ such that $F_k(x,a,b,c)\to x^{[2]}$. Hence we have {\bf
S(2)}.

Now we assume {\bf S(d)} holds for $d\ge 1$. Let ${\bf x}\in
\Q^{[d+1]}[x]$. We may assume that ${\bf x}$ is
$\F^{[d+1]}$-minimal, or we replace it by some $\F^{[d+1]}$-minimal
point in its $\F^{[d+1]}$-orbit closure. Let ${\bf x}=({\bf x'},
{\bf x''})$, where ${\bf x'}, {\bf x''}\in X^{[d]}$. Then ${\bf
x'}\in \Q^{[d]}[x]$. By {\bf S(d)}, there is a sequence $F_k\in
\F^{[d]}$ such that $F_k {\bf x'}\to x^{[d]}$. Without loss of
generality, we assume that $F_k{\bf x''}\to {\bf w}, k\to \infty$.
Then
$$(F_k\times F_k){\bf x}=(F_k\times F_k) ({\bf
x'}, {\bf x''})\to (x^{[d]},{\bf w})\in \Q^{[d+1]}(X), \ k\to
\infty.$$ Since $F_k\times F_k\in \F^{[d+1]}$ and ${\bf x}$ is
$\F^{[d+1]}$-minimal, $(x^{[d]}, {\bf w})$ is also
$\F^{[d+1]}$-minimal. By Lemma \ref{host-trick}, $(x^{[d]},{\bf
w})\in \overline{\F^{[d+1]}}(x^{[d+1]})$. Since
$(\overline{\F^{[d+1]}}(x^{[d+1]}),\F^{[d+1]})$ is minimal by
Theorem \ref{main1}-(1), we have $x^{[d+1]}$ is in the
$\F^{[d+1]}$-orbit closure of ${\bf x}$. Hence we have {\bf S(d+1)},
and the proof of claim is completed.

Since $x^{[d]}\in \overline{\F^{[d]}}({\bf x})$ for all ${\bf x}\in
\Q^{[d]}[x]$ and $(\overline{\F^{[d]}}(x^{[d]}), \F^{[d]})$ is
minimal, it is easy to see that $(\overline{\F^{[d]}}(x^{[d]}),
\F^{[d]})$ intersects all $\F^{[d]}$-minimal sets in $\Q^{[d]}[x]$
and hence it is the unique $\F^{[d]}$-minimal set in $\Q^{[d]}[x]$.
The proof is completed.
\end{proof}

\section{Lifting $\RP^{[d]}$ from factors to extensions}\label{section-lift}

In this section, first we give some equivalent conditions for
$\RP^{[d]}$, and give the proof of Theorem \ref{lifting}-(2), i.e.
lifting $\RP^{[d]}$ from factors to extensions.

\subsection{Equivalent conditions for $\RP^{[d]}$}

In this subsection we collect some equivalent conditions for
$\RP^{[d]}$.

\begin{prop}\label{prop-R}
Let $(X, T)$ be a minimal system and $d\in \N$. Then the following
conditions are equivalent:
\begin{enumerate}
\item $(x,y)\in \RP^{[d]}$;
\item $(x,y,y,\ldots,y)=(x,y^{[d+1]}_*) \in
\overline{\F^{[d+1]}}(x^{[d+1]});$
\item $(x,x^{[d]}_*, y, x^{[d]}_*) \in
\overline{\F^{[d+1]}}(x^{[d+1]})$.
\end{enumerate}
\end{prop}

\begin{proof}
By Theorem \ref{main2}, we have $(1)\Leftrightarrow(2)$. By Lemma
\ref{RPd-def} we have $(3)\Rightarrow (1)$. Now show
$(2)\Rightarrow(3)$.

If (2) holds, then $(x,y,y,\ldots,y)=(x,y^{[d+1]}_*) \in
\overline{\F^{[d+1]}}(x^{[d+1]})$ and $(x,y)\in \RP^{[d]}$. Since
$(x,y)\in \RP^{[d]}\subset \RP^{[d-1]}$, $(x,y^{[d]}_*)\in
\overline{\F^{[d]}}(x^{[d]})$. By Theorem \ref{main1},
$(\overline{\F^{[d]}}(x^{[d]}),\F^{[d]})$ is minimal. There is some
sequence $F_k\in\F^{[d]}$ such that $F_k(x, y^{[d]}_*)\to x^{[d]}, \
k\to \infty$. Then $$F_k\times F_k(x,y^{[d]}_*,y,y^{[d]}_*)\to
(x,x^{[d]}_*,y,x_*^{[d]}), \ k\to \infty.$$ Thus we have $(3)$, and
the proof is completed.
\end{proof}

\begin{lem}\label{RPd-lemma2}  Let $(X, T)$ be a minimal system. Then
$(x,y)\in \RP^{[d]}(X)$ if and only if $(x,x,\ldots,x,y) \in
\Q^{[d+1]}.$
\end{lem}

\begin{proof}
If $(x,y)\in \RP^{[d]}$, then by Proposition \ref{prop-R}, we have
$(x,x^{[d]}_*, y, x^{[d]}_*)$ $=$ $(x^{[d]}, y, x^{[d]}_*) $ $\in
\Q^{[d+1]} $. Since $\Q^{[d+1]}$ is invariant under the Euclidean
permutation of $X^{[d+1]}$, we have $(x,x,\ldots, x, y)\in
\Q^{[d+1]}$.

Conversely, assume that $(x,x,\ldots,x,y) \in \Q^{[d+1]}$. Since
$\Q^{[d+1]}$ is invariant under the Euclidean permutation of
$X^{[d+1]}$, we have $(x,x^{[d]}_*, y, x^{[d]}_*) \in \Q^{[d+1]}$.
This means that $(x,y)\in \RP^{[d]}$ by Lemma \ref{RPd-def}.
\end{proof}

\subsection{Lifting $\RP^{[d]}$ from factors to extensions}

In this section we will show Theorem \ref{lifting}-(2). First we
need a lemma.

\begin{lem}\label{proximal-lemma2}
Let $\pi: (X,T)\rightarrow (Y,T)$ be an extension of minimal
systems. If $(y_1,y_2)\in {\bf P}(Y,T)$ and $x_1\in \pi^{-1}(y_1)$
then there exists $x_2\in\pi^{-1}(y_2)$ such that $(x_1,x_2)\in {\bf
P}(X,T)$.
\end{lem}

\begin{proof}
Since $(y_1,y_2)\in {\bf P}(Y,T)$, by Proposition
\ref{proximal-Ellis} there is an minimal idempotent $u\in E(X,T)$
such that $uy_1=uy_2=y_2$. Let $x_2=ux_1$, then $\pi(x_2)=uy_1=y_2$.
By Proposition \ref{proximal-Ellis} $(x_1,x_2)\in {\bf P}(X,T)$ and
$\pi\times \pi(x_1,x_2)=(y_1,y_2)$.
\end{proof}

\begin{thm}
Let $\pi:(X,T)\rightarrow (Y,T)$ be an extension of minimal systems.
If $(y_1,y_2)\in \RP^{[d]}(Y)$, then there is
$(z_1,z_2)\in\RP^{[d]}(X)$ such that $$\pi\times
\pi(z_1,z_2)=(y_1,y_2).$$
\end{thm}

\begin{proof}
First we claim that it is sufficient to show the result when
$(y_1,y_2)$ is a minimal point of $(Y\times Y, T\times T)$. As a
matter of fact, by Proposition \ref{proximal-Ellis} there is a
minimal point $(y_1',y_2')\in \overline{\O((y_1,y_2),T\times T)}$
such that $(y_1',y_2')$ is proximal to $(y_1,y_2)$. Now
$(y_1',y_2')$ is minimal and $(y_1',y_2')\in \RP^{[d]}(Y)$. If we
have the claim already, then there is $(x_1',x_2')\in \RP^{[d]}(X)$
with $\pi\times \pi(x_1',x_2')=(y_1',y_2')$. Since $(y_1,y_1'),
(y_2,y_2')\in {\bf P}(Y,T)$, then by Lemma \ref{proximal-lemma2}
there are $x_1,x_2\in X$ with $\pi\times \pi(x_1,x_2)=(y_1,y_2)$
such that $(x_1',x_1)$, $(x_2',x_2)$ $\in$ ${\bf P}(X,T)$. This
implies that $(x_1,x_2)\in \RP^{[d]}(X)$ by Theorem \ref{main3}.
Hence we have the result for general case.

So we may assume that $(y_1,y_2)$ is a minimal point of $(Y\times Y,
T\times T)$. To make the idea of the proof clearer, we show the case
for $d=1$ first (see Figure 1). Since $(y_1,y_2)\in \RP^{[1]}(Y)$,
by Proposition \ref{prop-R} $(y_1,y_1,y_2,y_1)\in
\overline{\F^{[2]}}(y_1^{[2]})$. So there is some sequence $F_k\in
\F^{[2]}$ such that $$F_k {y_1^{[2]}}\to (y_1, y_1,y_2,y_1),\ k\to
\infty .$$ Take a point $x_1\in \pi^{-1}(y_1)$. Without loss of
generality, we may assume that
\begin{equation*}
    F_k {x_1^{[2]}}\to (x_1,x_2,x_3,x_4), \ k\to \infty.
\end{equation*}
Then $\pi^{[2]}(x_1,x_2,x_3,x_4)=(y_1,y_1,y_2,y_1)$. Take
$\{n_k\}\subset \Z$ such that $T^{n_k}x_2\to x_1, k\to \infty$ and
assume that $T^{n_k}x_4\to x_4', k\to \infty$. Then
$$({\rm id}\times T\times {\rm id}\times T
)^{n_k}(x_1,x_2,x_3,x_4)\to (x_1,x_1,x_3,x_4'), \ k\to \infty .$$
Since ${\rm id}\times T\times {\rm id}\times T =T^{[2]}_1 \in
\F^{[2]}$, we have $(x_1,x_1,x_3,x_4')\in
\overline{\F^{[2]}}(x_1^{[2]}).$ Now take $\{m_k\}\subset \Z$ such
that $T^{m_k}x_3\to x_1, k\to \infty$ and assume that
$T^{m_k}x_4'\to x_4'', k\to \infty$. Then
$$({\rm id}\times {\rm id}\times T\times T
)^{m_k}(x_1,x_1,x_3,x_4')\to (x_1,x_1,x_1,x_4''), \ k\to \infty .$$
Since ${\rm id}\times {\rm id}\times T\times T =T^{[2]}_2 \in
\F^{[2]}$, we have $(x_1,x_1,x_1,x_4'')\in
\overline{\F^{[2]}}(x_1^{[2]}).$ By Lemma \ref{RPd-lemma2}
$(x_1,x_4'')\in \RP^{[1]}(X)$. Let $y_3=\pi(x_4'')$. Note that
$(x_1,x_4'')\in \overline{\O((x_3,x_4'),T\times T)}$, and we have
$(y_3,y_1)\in \overline{\O((y_1,y_2),T\times T)}$. Since $(y_1,y_2)$
is $T\times T$-minimal, there is a sequence $\{a_k\}\subset \Z$ such
that $(T\times T)^{a_k}(y_3,y_1)\to (y_1,y_2), k\to \infty$. Without
loss of generality, we may assume that there are $z_1,z_2\in X$ such
that
$$(T\times T)^{a_k}(x_4'',x_1)\to (z_1,z_2), \ k\to \infty$$
Since $(x_1,x_4'')\in \RP^{[1]}(X)$ and $\RP^{[1]}(X)$ is closed and
invariant, we have $(z_1,z_2)\in \overline{\O((x_4'',x_1),T\times
T)}\subset \RP^{[1]}(X)$. Note that
$$\pi\times \pi(z_1,z_2)=\lim_k(T\times T)^{a_k}(\pi(x_4''),\pi(x_1))=
\lim_k (T\times T)^{a_k}(y_3,y_1)=(y_1,y_2),$$ we are done for the
case $d=1$. For the proof when $d=2$, see Figure 2.

\medskip

\vspace{2mm}
\begin{center}

\psset{unit=0.9cm}

\begin{pspicture}(-1,-3.5)(11,6)

\psline[linewidth=0.5pt](2, 0.0)(2,5.5)

\psline[linewidth=0.5pt](3.5, 0.0)(3.5,5.5)

\psline[linewidth=0.5pt](0,-0.5)(4,-0.5)

\qdisk(2,-0.5){0.05}\put(1.75,-0.8){$y_1$}

\qdisk(3.5,-0.5){0.05}\put(3.25,-0.8){$y_2$}

\qdisk(2,1){0.05}\put(1.5,1){$x_4$}

\qdisk(2,2){0.05}\put(1.5,2){$x_4'$}

\qdisk(2,3){0.05}\put(1.5,3){$x_2$}

\qdisk(2,4){0.05}\put(1.5,4){$x_1$}

\qdisk(3.5,4){0.05}\put(3.75, 4){$x_3$}

\pscurve[linewidth=1pt]{->}(2.1, 3)(2.75,3.5)(2.1,4)

\pscurve[linewidth=1pt]{->}(2.1, 1)(2.75,1.5)(2.1,2)

\put(1,-1.5){$(x_1,x_2,x_3,x_4)$}

\put(2.25,-2){$\downarrow$}

\put(1,-2.5){$(x_1,x_1,x_3,x_4')$}


\psline[linewidth=0.5pt](7, 0.0)(7,5.5)

\psline[linewidth=0.5pt](8.5, 0.0)(8.5,5.5)

\psline[linewidth=0.5pt](10, 0.0)(10,5.5)

\psline[linewidth=0.5pt](6.5,-0.5)(11,-0.5)

\qdisk(8.5,-0.5){0.05} \put(8.25,-0.8){$y_1$}

\qdisk(10,-0.5){0.05}  \put(9.75,-0.8){$y_2$}

\qdisk(7,-0.5){0.05}  \put(6.75,-0.8){$y_3$}

\qdisk(8.5,2){0.05}\put(8,1.5){$x_4'$}

\qdisk(8.5,1){0.05}\put(8.6,1){$z_1$}

\qdisk(8.5,4){0.05}\put(8,4){$x_1$}

\qdisk(10,4){0.05}\put(10.25, 4){$x_3$}

\qdisk(10,3){0.05}\put(10.25, 3){$z_2$}

\qdisk(7,2){0.05}\put(6.5,1.5){$x_4''$}

\pscurve[linewidth=1pt]{->}(9.9, 4)(9.25,4.5)(8.6,4)

\pscurve[linewidth=1pt]{->}(8.4, 2)(7.75,2.5)(7.1,2)

\pscurve[linewidth=1pt,linestyle=dashed]{->}(7.1,2)(8,1.1)(8.4,1)

\pscurve[linewidth=1pt,linestyle=dashed]{->}(8.6,4)(9,3.5)(9.9,3)

\put(7.5,-1.5){$(x_1,x_1,x_3,x_4')$}

\put(8.75,-2){$\downarrow$}

\put(7.5,-2.5){$(x_1,x_1,x_1,x_4'')$}

\put(3,-3.5){Figure 1. \ The case $d=1$ }

\end{pspicture}
\end{center}
\vspace{2mm}

\begin{center}

\psset{unit=0.9cm}

\begin{pspicture}(-1,-4.5)(13,6.5)

\psline[linewidth=0.5pt](1, 0.0)(1,5.5)

\psline[linewidth=0.5pt](2.5, 0.0)(2.5,5.5)

\psline[linewidth=0.5pt](0,-0.5)(4,-0.5)

\qdisk(1,-0.5){0.05}\put(0.75,-0.8){$y_1$}

\qdisk(2.5,-0.5){0.05}\put(2.25,-0.8){$y_2$}

\qdisk(1,0.5){0.05}\put(0.,0.5){$x_{111}$}

\qdisk(1,1){0.05}\put(0.,1){$x_{111}'$}

\qdisk(1,1.5){0.05}\put(0.,1.5){$x_{011}$}

\qdisk(1,2){0.05}\put(0.,2){$x_{011}'$}

\qdisk(1,2.5){0.05}\put(0.,2.5){$x_{101}$}

\qdisk(1,3){0.05}\put(0.,3){$x_{101}'$}

\qdisk(1,3.5){0.05}\put(0.,3.5){$x_{110}$}

\qdisk(1,4){0.05}\put(0.,4){$x_{010}$}

\qdisk(1,4.5){0.05}\put(0.,4.5){$x_{100}$}

\qdisk(1,5){0.05}\put(-1.,5){$x_{000}=x_1$}

\qdisk(2.5,4){0.05}\put(2.75, 4){$x_{001}$}

\pscurve[linewidth=1pt]{->}(1.1, 0.5)(1.75,0.75)(1.1,1)

\pscurve[linewidth=1pt]{->}(1.1, 1.5)(1.75,1.75)(1.1,2)

\pscurve[linewidth=1pt]{->}(1.1, 2.5)(1.75,2.75)(1.1,3)

\pscurve[linewidth=1pt]{->}(1.1, 4.5)(1.75,4.75)(1.1,5)

\pscurve[linewidth=1pt]{->}(1.1, 4)(2,4.75)(1.1,5)

\pscurve[linewidth=1pt]{->}(1.1, 3.5)(2.25,4.75)(1.1,5)

\put(-1.5,-1.5){{\footnotesize
$(x_{000},x_{100},x_{010},x_{110},x_{001},x_{101},x_{011},x_{111})$}}

\put(2.25,-2){$\downarrow$}

\put(-1,-2.5){{\footnotesize$(x_{1},x_{1},x_{1},x_{1},x_{001},x'_{101},x'_{011},x'_{111})$}}


\psline[linewidth=0.5pt](7, 0.0)(7,6.5)

\psline[linewidth=0.5pt](8.5, 0.0)(8.5,6.5)

\psline[linewidth=0.5pt](10, 0.0)(10,6.5)

\psline[linewidth=0.5pt](11.5, 0.0)(11.5,6.5)

\psline[linewidth=0.5pt](6.5,-0.5)(12.5,-0.5)

\qdisk(8.5,-0.5){0.05} \put(8.25,-0.8){$y_1$}

\qdisk(10,-0.5){0.05}  \put(9.75,-0.8){$y_2$}

\qdisk(7,-0.5){0.05}  \put(6.75,-0.8){$y_3$}

\qdisk(11.5,-0.5){0.05}  \put(11.25,-0.8){$y_4$}

\qdisk(10,5){0.05}\put(10.25, 5){$x_{001}$}

\qdisk(10,2){0.05}\put(10.25, 2){$z_2$}

\qdisk(8.5,5){0.05}\put(8,5){{$x_1$}}

\qdisk(8.5,6){0.05}\put(8.6,6){{$z_1$}}

\qdisk(8.5,4){0.05}\put(7.5,4){$x_{101}'$}

\qdisk(8.5,3){0.05}\put(8.7,3){$x_{011}'$}

\qdisk(8.5,2){0.05}\put(7.5,2){$x_{111}'$}

\qdisk(8.5,1){0.05}\put(7.5,1){$x_{111}''$}

\qdisk(7,3){0.05}\put(6.,3){$x_{011}''$}

\qdisk(11.5,1){0.05}\put(11.6,1){$x_{111}'''$}

\pscurve[linewidth=1pt]{->}(9.9, 5)(9.25,5.3)(8.6,5)

\pscurve[linewidth=1pt]{->}(8.6, 4)(9.25,4.5)(8.6,4.9)

\pscurve[linewidth=1pt]{->}(8.6, 2)(9.25,1.5)(8.6,1)

\pscurve[linewidth=1pt]{->}(8.4, 3)(7.75,3.5)(7.1,3)

\pscurve[linewidth=1pt]{->}(8.4, 3)(7.75,3.5)(7.1,3)

\pscurve[linewidth=1.5pt]{->}(8.6,1)(10,0.5)(11.4,1)

\pscurve[linewidth=1.5pt]{->}(6.8,3)(6.3,4.8)(8,5.)

\pscurve[linewidth=1pt,linestyle=dashed]{->}(11.4,1)(11,1.5)(10.1,2)

\pscurve[linewidth=1pt,linestyle=dashed]{->}(8.4,5.2)(7.5,5.9)(8.4,6)

\put(6.5,-1.5){{\footnotesize$(x_{1},x_{1},x_{1},x_{1},x_{001},x'_{101},x'_{011},x'_{111})$}}

\put(8.75,-2){$\downarrow$}

\put(6.5,-2.5){{\footnotesize$(x_{1},x_{1},x_{1},x_{1},x_{1},x_{1},x''_{011},x''_{111})$}}

\put(8.75,-3){$\downarrow$}

\put(6.5,-3.5){{\footnotesize$(x_{1},x_{1},x_{1},x_{1},x_{1},x_{1},x_{1},x'''_{111})$}}

\put(3,-4.5){Figure 2. \ The case $d=2$ }

\end{pspicture}
\end{center}
\vspace{2mm} The idea of the proof in the general case is the
following. For a point ${\bf x}\in \overline{\F^{[d+1]}}(x_1)$ we
apply face transformations $F_1^k$ such that the first
$2^d$-coordinates of ${\bf x}_1=\lim F_1^k\bf x$ will be
$x^{[d]}_1$. Then apply face transformations $F_2^k$ such that the
first $2^d+2^{d-1}$-coordinates of ${\bf x}_2=\lim F^k_2{\bf x}_1$
will be $(x^{[d]}_1, x^{[d-1]}_1)$. Repeating this process we get a
point $((x_1^{[d+1]})_*,x_2)\in \overline{\F^{[d+1]}}(x_1)$ which
implies that $(x_1, x_2)\in \RP^{[d]}(X)$. Then we use the same idea
used in the proof when $d=1,2$ to trace back to find $(z_1,z_2)$.
Here are the details.

\medskip

Now let $(y_1, y_2)\in \RP^{[d]}(Y)$, then by Proposition
\ref{prop-R}, $(y_1^{[d]},y_2,(y_1^{[d]})_*)\in
\overline{\F^{[d+1]}}(y_1^{[d+1]})$. So there is some sequence
$F_k\in \F^{[d+1]}$ such that $$F_k {y_1^{[d+1]}}\to
(y_1^{[d]},y_2,(y_1^{[d]})_*),\ k\to \infty .$$ Without loss of
generality, we may assume that
\begin{equation}\label{d10}
    F_k {x_1^{[d+1]}}\to {\bf x}, \ k\to \infty.
\end{equation}
Then $x_\emptyset=x_1$ and $\pi^{[d+1]}({\bf
x})=(y_1^{[d]},y_2,(y_1^{[d]})_*)$.

Let ${\bf x_I}=(x_\ep: {\ep(d+1)=0})\in X^{[d]}$ and ${\bf
x_{II}}=(x_\ep: {\ep(d+1)=1})\in X^{[d]}$. Then ${\bf x}=({\bf x_I},
{\bf x_{II}})$. Note that
$$\pi^{[d]}({\bf x_I})=\pi^{[d]}(x_1^{[d]})=y_1^{[d]}, \ \text{and}\ \pi^{[d]}({\bf x_{II}})=
(y_2, (y^{[d]}_1)_*).$$ By Proposition \ref{unique-minimal-set},
there is some sequence $F^1_k\in \F^{[d]}$ such that
$$F_k^1({\bf x_I})\to x_1^{[d]} , \ k\to \infty .$$
We may assume that
$$ F_k^1({\bf x_{II}})\to {\bf x'_{II}},\ k\to \infty .$$
Note that $\pi^{[d]}({\bf x_{II}})=\pi^{[d]}( {\bf x'_{II}})=(y_2,
(y^{[d]}_1)_*)$.

Let $F^1_k=(S_{\ep'}^k: \ep'\in \{0,1\}^d)$. Let $H_k^1=(S_\ep^k:
\ep\in \{0,1\}^{d+1})\in \F^{[d+1]}$ such that
$$(S_\ep^k: \ep\in \{0,1\}^{d+1}, \ep(d+1)=0)=(S_\ep^k: \ep\in \{0,1\}^{d+1}, \ep(d+1)=1)=F_k^1.$$
Then $$H^1_k({\bf x})=F^1_k\times F^1_k({\bf x_I},{\bf x_{II}})\to
(x_1^{[d]}, {\bf x_{II}'})\triangleq {\bf x^1}\in \overline
{\F^{[d+1]}}(x_1^{[d+1]}), \ k\to \infty.$$ Let ${\bf
y^1}=\pi^{[d+1]}({\bf x^1})$. It is easy to see that $x^1_\ep=x_1$
if $\ep(d+1)=0$. For ${\bf y^1}$, $y^1_{\{d+1\}}=y^1_{00\ldots
01}=y_2$ and $y^1_{\ep}=y_1$ for all $\ep\neq \{d+1\}$.

Let ${\bf x^1_I}=(x_\ep: \ep\in \{0,1\}^{d+1}, \ep(d)=0)\in X^{[d]}$
and ${\bf x^1_{II}}=(x_\ep: \ep\in \{0,1\}^{d+1}, \ep(d)=1)\in
X^{[d]}$. By Proposition \ref{unique-minimal-set}, there is some
sequence $F^2_k\in \F^{[d]}$ such that
$$F_k^2({\bf x^1_I})\to x_1^{[d]}, \ F_k^2({\bf x^1_{II}})\to {\bf {x^1_{II}}'}, k\to \infty$$
and $\pi^{[d]}( {\bf {x^1_{II}}'})=(y^{[d-1]}_1, y_3,
(y^{[d-1]}_1)_*)$ for some $y_3\in Y$.

Let $F^2_k=(S_{\ep'}^k: \ep'\in \{0,1\}^d)$. Let $H_k^2=(S_\ep^k:
\ep\in \{0,1\}^{d+1})\in \F^{[d+1]}$ such that
$$(S_\ep^k: \ep\in \{0,1\}^{d+1}, \ep(d)=0)=(S_\ep^k: \ep\in \{0,1\}^{d+1}, \ep(d)=1)=F_k^2.$$
Then let $$H^2_k({\bf x^1})\to {\bf x^2}\in \overline
{\F^{[d+1]}}(x_1^{[d+1]}),\ k\to\infty.$$ Let ${\bf
y^2}=\pi^{[d+1]}({\bf x^2})$. Then $H^2_k({\bf y^1})\to {\bf y^2},\
k\to\infty.$ From this one has that $(y_3,y_1)\in
\overline{\O((y_1,y_2),T\times T)}$. By the definition of ${\bf
x^2}, {\bf y^2}$, it is easy to see that $x^2_\ep=x_1$ if
$\ep(d+1)=0$ or $\ep(d)=0$; $y^2_{\{d,d+1\}}=y^2_{00\ldots 011}=y_3$
and $y^2_{\ep}=y_1$ for all $\ep\neq \{d, d+1\}$.

Now assume that we have ${\bf x^j}\in \overline
{\F^{[d+1]}}(x_1^{[d+1]})$ for $1\le j\le d$ with $\pi^{[d+1]}({\bf
x^j})={\bf y^j}$ such that $x^j_\ep=x_1$ if there exists some $k$
with $d-j+2\le k\le d+1$ such that $\ep(k)=0$; $y^j_{\{d-j+2,\ldots,
d,d+1\}}=y_{j+1}$ and $y^j_{\ep}=y_1$ for all $\ep\neq
\{d-j+2,\ldots,d, d+1\}$, and $(y_{j+1},y_1)\in
\overline{\O((y_1,y_{j}),T\times T)}$.

Let ${\bf x^j_I}=(x_\ep: \ep\in \{0,1\}^{d+1}, \ep(d-j+1)=0)\in
X^{[d]}$ and ${\bf x^j_{II}}=(x_\ep: \ep\in \{0,1\}^{d-j+1},
\ep(d-j+1)=1)\in X^{[d]}.$ By Proposition \ref{unique-minimal-set},
there is some sequence $F^{j+1}_k\in \F^{[d]}$ such that
$$F_k^{j+1}({\bf x^j_I})\to x_1^{[d]}, \ F_k^{j+1}({\bf x^j_{II}})\to {\bf {x^j_{II}}'}, k\to \infty.$$

Let $F^{j+1}_k=(S_{\ep'}^k: \ep'\in \{0,1\}^d)$. Let
$H_k^{j+1}=(S_\ep^k: \ep\in \{0,1\}^{d+1})\in \F^{[d+1]}$ such that
$$(S_\ep^k: \ep\in \{0,1\}^{d+1}, \ep(d-j+1)=0)=(S_\ep^k: \ep\in \{0,1\}^{d+1}, \ep(d-j+1)=1)=F_k^{j+1}.$$
Then let $$H^{j+1}_k({\bf x^j})\to {\bf x^{j+1}}\in \overline
{\F^{[d+1]}}(x_1^{[d+1]}),\ k\to \infty .$$ It is easy to see that
$x^{j+1}_\ep=x_1$ if there exists some $k$ with $d-j+1\le k\le d+1$
such that $\ep(k)=0$.

Let ${\bf y^{j+1}}=\pi^{[d+1]}({\bf x^{j+1}})$. Then
$y^{j+1}_{\ep}=y_1$ for all $\ep\neq \{d-j+1, d-j+2,\ldots,d+1\}$,
and denote $y^j_{\{d-j+1,d-j+2,\ldots ,d+1\}}=y_{j+2}$. Note that
$H^2_k({\bf y^j})\to {\bf y^{j+1}},\ k\to\infty.$ From this one has
that $(y_{j+2},y_1)\in \overline{\O((y_1,y_{j+1}),T\times T)}$.

\medskip

Inductively we get ${\bf x^1, \ldots, x^{d+1}}$ and ${\bf y^1,
\ldots, y^{d+1}}$ such that for all $1\le j\le d+1$ ${\bf x^j}\in
\overline {\F^{[d+1]}}(x_1^{[d+1]})$ with $\pi^{[d+1]}({\bf
x^j})={\bf y^j}$. And $x^j_\ep=x_1$ if there exists some $k$ with
$d-j+2\le k\le d+1$ such that $\ep(k)=0$; $y^j_{\{d-j+2,\ldots,
d,d+1\}}=y_{j+1}$ and $y^j_{\ep}=y_1$ for all $\ep\neq
\{d-j+2,\ldots,d, d+1\}$, and $(y_{j+1},y_1)\in
\overline{\O((y_1,y_{j}),T\times T)}$.

For ${\bf x^{d+1}}$, we have that $x^{d+1}_\ep=x_1$ if there exists
some $k$ with $1\le k\le d+1$ such that $\ep(k)=0$. That means there
is some $x_2\in X$ such that
$${\bf x^{d+1}}=(x_1,x_1,\ldots, x_1, x_2)\in \overline
{\F^{[d+1]}}(x_1^{[d+1]}). $$ By Lemma \ref{RPd-lemma2},
$(x_1,x_2)\in \RP^{[d]}(X)$. Note that $\pi(x_2)=y_{d+2}$.

Since $(y_{j+1},y_1)\in \overline{\O((y_1,y_{j}),T\times T)}$ for
all $1\le j \le d+1$, we have $(y_{d+2},y_1)\in
\overline{\O((y_1,y_2),T\times T)}$ or $(y_1, y_{d+2})\in
\overline{\O((y_1,y_2),T\times T)}$. Without loss of generality, we
assume that $(y_1,y_{d+2})\in \overline{\O((y_1,y_2),T\times T)}$.
Since $(y_1,y_2)$ is $T\times T$-minimal, there is some
$\{n_k\}\subset \Z$ such that $(T\times T)^{n_k}(y_1,y_{d+2})\to
(y_1,y_2), k\to \infty$. Without loss of generality, we assume that
$$(T\times T)^{n_k}(x_1,x_2)\to (z_1,z_2), \ k\to \infty.$$ Since
$\RP^{[d]}(X)$ is closed and invariant, we have
$$(z_1,z_2)\in \overline{\O((x_1,x_2),T\times T)}\subset
\RP^{[d]}(X).$$ And $$\pi\times \pi(z_1,z_2)=\lim_k(T\times
T)^{n_k}(\pi(x_1),\pi(x_2))= \lim_k (T\times
T)^{n_k}(y_1,y_{d+2})=(y_1,y_2).$$ The proof is completed.
\end{proof}

\section{A combinatorial consequence and group actions}\label{section-appl}

\subsection{A combinatorial consequence} We have the following
combinatorial consequence of the fact that
$(\overline{\F^{[d]}}(x^{[d]}), \F^{[d]})$ is minimal.





\begin{prop} Let $(X,T)$ be a minimal system, $x\in X$ and $U$ be an open
neighborhood of $x$. Put $S=\{n\in\Z:T^nx\in U\}$. Then for each
$d\ge 1$,
$$\{(n_1,\ldots,n_d)\in \Z^d: n_1\ep_1+\cdots +n_d\ep_d\in S,
\ep_i\in \{0,1\}, 1\le i\le d\}$$ is syndetic.
\end{prop}
\begin{proof} This follows by that fact that $x^{[d]}$ is a minimal
point under the face group action $\F^{[d]}$.
\end{proof}

To understand $S$ better we show the following proposition which is
similar to \cite[Proposition 2.3]{HY}. Note that a collection $\F$
of subsets of $\Z$ is a {\it family} if it is upwards, i.e. $A\in
\F$ and $A\subset B$ imply that $B\in \F$.

\begin{prop} The family of dynamically syndetic subsets is the family generated by the sets $S$ whose
indicator functions $1_S$ are the minimal points of
$(\{0,1\}^\Z,\sigma)$ and $0\in S$, where $\sigma$ is the shift.
\end{prop}

\begin{proof} Put $\Sigma=\{0,1\}^\Z$. We denote the family generated by the sets containing $\{0\}$ whose
indicator functions are the minimal points of $(\Sigma,\sigma)$ by
$\F_m$. Clearly, if $1_F$ is the indicator function of $F$ then
$F=N(1_F,[1])$, where $[1]=\{ s\in \Sigma: s(0)=1 \}$. Hence $\F_m$
is contained in the family of dynamical syndetic subsets.

On the other hand, let $A$ be a dynamical syndetic subset. Then
there exist a minimal system $(X,T)$ with metric $d$, $x\in X$ and
an open neighborhood $V$ of $x$ such that $A\supset N(x,V)=\{n\in\Z:
T^nx\in V\}$. It is easy to see that we can shrink $V$ to an open
neighborhood $V'$ of $x$ whose boundary is disjoint from the orbit
of $x$.

Then do the classical lifting trick, a la Glasner, Adler etc. Let
$$
Y=\{ (z,t)\in X\times \Sigma : t(i)=1 \text{ implies } T^i z \in
\text{cl}(V')\text{ and } t(i)=0 \text{ implies } T^iz\in
\text{cl}(X\setminus V') \}$$

Then $Y$ is a $T\times \sigma$-invariant closed subset of $X\times
\Sigma$. Since the orbit of $x$ doesn't meet the boundary of $V'$,
there is a unique $t\in \Sigma$ such that $(x,t)\in Y$ and $t$ is
the indicator function of $N(x,V')$. Take a minimal subset $J$ of
$(Y,T\times \sigma)$ with $J\subset \overline{\O((x,t),T\times
\sigma)}$ and let $\pi_X:J\rightarrow X$ be the projective map.
Since $(X,T)$ is minimal, $\pi_X(J)=X$. Hence $(x,t)\in J$.
Projecting $J$ to $\Sigma$ we see that $t$ is a minimal point. Hence
$A\in \F_m$ as $A\supset N(x,V')$ and $t=1_{N(x,V')}$.
\end{proof}

\begin{rem} We note that if $S$ is a syndetic subset of $\Z$ then
$S-S\supset S_1-S_1$ for some dynamically syndetic subset $S_1$.
\end{rem}

\subsection{Abelian group actions}

\begin{de}
Let $X$ be a compact metric space, $G$ be an abelian topological
group acting on $X$ and let $d\ge 1$ be an integer. A pair $(x, y)
\in X\times X$ ia said to be {\em regionally proximal of order $d$
of $G$-action} if for any $\d
> 0$, there exist $x', y'\in X$ and a vector ${\bf n} = (n_1,\ldots
, n_d)\in G^d$ such that $\rho(x, x') < \d, \rho(y, y') <\d$, and $$
\rho(T^{{\bf n}\cdot \ep}x', T^{{\bf n}\cdot \ep}y') < \d\ \text{for
any nonempty $\ep\subset [d]$},$$ where ${\bf n}\cdot \ep =
\sum_{i\in \ep} n_i$. The set of regionally proximal pairs of order
$d$ of $G$-action is denoted by $\RP_G^{[d]}(X)$, which is called
{\em the regionally proximal relation of order $d$ of $G$-action}.
\end{de}

A subset $S\subseteq G$ is a {\em central set} if there exists a
system $(X,G)$, a point $x\in X$ and a minimal point $y$ proximal to
$x$, and a neighborhood $U_y$ of $y$ such that $N(x,U_y)\subset S$.
The notion of IP-set can be defined in this setting too. By the
proof of Furstenberg \cite[Proposition 8.10.]{F} we have

\begin{lem}
Let $G$ be an abelian group. Then any central set is an IP-set.
\end{lem}

So we have
\begin{lem} If $(X,G)$ is minimal, then ${\bf P}(X)\subset\RP_G^{[d]}(X)$.
\end{lem}

At the same time the notions of face group and parallelepiped group
can be defined. So we have the following theorem by our proof

\begin{thm} Let $(X,G)$ a minimal system with $G$ being abelian. Then
$\RP_G^{[d]}(X)$ is a closed invariant equivalence relation. So
$(X/\RP_G^{[d]}(X), G)$ is distal.
\end{thm}

Similar to \cite{HKM} we may define
\begin{de}
Let $(X,G)$ a minimal system with $G$ being abelian. We call
$(X/\RP_G^{[d]}(X), G)$ the $d$-step nilfactor for $G$-action.
\end{de}

We think that it will be interesting to study the properties of
$(X/\RP_G^{[d]}(X), G)$ or more general group actions.

\appendix

\section{Basic facts about abstract topological
dynamics}\label{section-tds}

In this appendix we recall some basic definitions and results in
abstract topological systems. For more details, see \cite{Au88, E69,
G76, G96, V77, Vr}.

\subsection{Topological transformation groups}

A {\em topological dynamical systems} is a triple $\X=(X, \t, \Pi)$,
where $X$ is a compact $T_2$ space, $\t$ is a $T_2$ topological
group and $\Pi: T\times X\rightarrow X$ is a continuous map such
that $\Pi(e,x)=x$ and $\Pi(s,\Pi(t,x))=\Pi(st,x)$. We shall fix $\t$
and suppress the action symbol. In lots of literatures, $\X$ is also
called a {\em topological transformation group} or a {\em flow}.
Usually we omit $\Pi$ and denote a system by $(X,\t)$.

\medskip

Let $(X,\t)$ be a system and $x\in X$, then $\O(x,\t)$ denotes the
{\em orbit} of $x$, which is also denoted by $\t x$. A subset
$A\subseteq X$ is called {\em invariant} if $t a\subseteq A$ for all
$a\in A$ and $t\in \t$. When $Y\subseteq X$ is a closed and
$\t$-invariant subset of the system $(X, \t)$ we say that the system
$(Y, \t)$ is a {\em subsystem} of $(X, \t)$. If $(X, \t)$ and $(Y,
\t)$ are two dynamical systems their {\em product system} is the
system $(X \times Y, \t)$, where $t(x, y) = (tx, ty)$.

\medskip

A system $(X,\t)$ is called {\em minimal} if $X$ contains no proper
closed invariant subsets. $(X,\t)$ is called {\em transitive} if
every invariant open subset of $X$ is dense. An example of an
transitive system is a {\em point-transitive} system, which is a
system with a dense orbit. It is easy to verify that a system is
minimal iff every orbit is dense. The system $(X,\t)$ is {\em weakly
mixing} if the product system $(X \times X,\t)$ is transitive.

\medskip

A {\em homomorphism} (or {\em extension}) of systems $\pi :
(X,\t)\rightarrow (Y,\t)$ is a continuous onto map of the phase
spaces such that $\pi(tx)=t\pi(x)$ for all $t\in \t, x\in X$. In
this case one says that $(Y,\t)$ is a {\em factor} of $(X,\t)$ and
also that $(X,\t)$ is an {\em extension} of $(Y,\t)$. Define
\begin{equation*}
    R_\pi=\{(x_1,x_2): \pi(x_1)=\pi(x_2)\},
\end{equation*}
then $Y=X/ R_\pi$. For $n\ge 2$, define
\begin{equation*}
    R_\pi^n=\{(x_1,x_2,\ldots, x_n)\in X^n: \pi(x_1)=\pi(x_2)=\ldots=\pi(x_n)\},
\end{equation*}
and let $R^1_\pi=X$.

\subsection{Enveloping semigroups}
Given a system $(X,\t)$ its {\em enveloping semigroup} or {\em Ellis
semigroup} $E(X,\t)$ is defined as the closure of the set $\{t: t\in
\t\}$ in $X^X$ (with its compact, usually non-metrizable, pointwise
convergence topology). For an enveloping semigroup, $E\rightarrow
E:$ $p\mapsto pq$ and $p\mapsto tp$ is continuous for all $q\in E$
and $t\in \t$. Note that $(X^X,\t)$ is a system and $(E(X,\t),\t)$
is its subsystem.

\medskip

Let $(X,\t),(Y,\t)$ be systems and $\pi: X\rightarrow Y$ be an
extension. Then there is a unique continuous semigroup homomorphism
$\pi^* : E(X,\t)\rightarrow E(Y,\t)$ such that
$\pi(px)=\pi^*(p)\pi(x)$ for all $x\in X,p\in E(X,\t)$. When there
is no confusion, we usually regard the enveloping semigroup of $X$
as acting on $Y$: $p\pi(x)=\pi(px)$ for $x\in X$ and $p\in E(X,\t)$.

\subsection{Idempotents and ideals}

For a semigroup the element $u$ with $u^2=u$ is called an {\it
idempotent}. Ellis-Namakura Theorem says that for any enveloping
semigroup $E$ the set $J(E)$ of idempotents of $E$ is not empty
\cite{E69}. A non-empty subset $I \subset E$  is a {\em left ideal}
(resp. {\em right ideal}) if it $EI \subseteq I$ (resp. {\em $IE
\subseteq I$}). A {\em minimal left ideal} is the left ideal that
does not contain any proper left ideal of $E$. Obviously every left
ideal is a semigroup and every left ideal contains some minimal left
ideal.

\medskip
We can introduce a quasi-order (a reflexive, transitive relation)
$<_L$ on the set $J(E)$  by defining $v <_L u $ if and only if
$vu=v$. If $v<_L u$ and $u <_L v$ we say that $u$ and $v$ are {\em
equivalent} and write $u \sim_L v$. Similarly, we define $<_R$ and
$\sim_R$. An idempotent $u \in J(E)$ is {\em minimal} if $v \in
J(E)$ and $v<_L u$ implies $u<_L v$. The following results are
well-known \cite{EEN, FK89}: let $L$ be a left ideal of enveloping
semigroup $E$ and $u \in J(E)$. Then there is some idempotent $v$ in
$Lu$ such that $v<_Ru$ and $v<_Lu$; an idempotent is minimal if and
only if it is contained in some minimal left ideal.

\medskip

Minimal left ideals have very rich algebraic properties. For
example,
\begin{prop}\label{left-ideal}
Let $I$ be a minimal left ideal, then
\begin{enumerate}
  \item $I=\bigcup_{u\in J(I)} uI$ is its partition and every $uI$ is a
group with identity $u\in J(I)$.
  \item All minimal
idempotents in the same minimal left ideal are equivalent to each
other, i.e. for all $u,v\in J(I)$, $u\sim_L v$.
\end{enumerate}

\end{prop}

\medskip
Let $(X,\t)$ be a system and $E(X,\t)$ be its enveloping semigroup.
A subset $I \subseteq E(X,\t)$ is a closed left ideal of $E(X,\t)$
iff $(I, \t)$ is a subsystem of $(E(X,\t),\t)$. And $I$ is a minimal
left ideal of $E(X,\t)$ iff $(I, \t)$ is minimal. Let $I\subset
E(X,\t)$ be a minimal left ideal. Then for all $x\in X$, $Ix=\{px:
p\in I\}$ is a minimal subset of $X$. Especially if $(X,\t)$ is
minimal itself, then $X=Ix$ for all $x\in X$. It follows that
\begin{prop}\label{minimal-point}
A point $x\in X$ is minimal if and only if $ux=x$ for some $u\in I$.
\end{prop}

\subsection{Universal point transitive system and universal minimal system}

For fixed $\t$,  there exists a universal point-transitive system
$\mathcal{S}_\t=(S_\t,\t)$ such that $\t$ can densely and
equivariantly be embedded in $S_\t$. The multiplication on $\t$ can
be extended to a multiplication on $S_\t$, then $S_\t$ is a closed
semigroup with continuous right translations. The universal minimal
system $\mathfrak{M}=(\M,\t)$ is isomorphic to any minimal left
ideal in $S_\t$ and $\M$ is a closed semigroup with continuous right
translations. Hence $J=J(\M)$ of idempotents in $\M$ is nonempty.
Moreover, $\{v\M:v\in J\}$ is a partition of $\M$ and every $v\M$ is
a group with unit element $v$. Sometimes if there are chances being
confusion then we will use $\M_\t$ instead of $\M$.

\medskip

The sets $S_\t$ and $\M$ act on $X$ as semigroups and $S_\t
x=\overline{\t x}$, while for a minimal system $(X,\t)$ we have $\M
x=\overline{\t x}=X$ for every $x\in X$. A necessary and sufficient
condition for $x$ to be minimal is that $ux=x$ for some $u\in J$.

\subsection{All kinds of extensions}

Two points $x_1$ and $x_2$ are called {\em proximal} iff
$$\overline{\t (x_1,x_2)}\cap \Delta_X\neq \emptyset.$$
Let $\U_X$ be the unique uniform structure of $X$, then
$${\bf P}={\bf P}(X,\t) = \bigcap \big\{ \t \a : \a\in \U_X \big\}$$
is the collection of proximal pairs in $X$, the {\em proximal
relation}.

\begin{prop}\label{proximal-Ellis}
Let $(X,\t)$ be a system. Then
\begin{enumerate}
  \item $x_1$, $x_2$ are proximal in $(X,\t)$ iff
$px_1 = px_2$ for some $p\in E(X,\t)$.
  \item If $x\in X$ and $u$ is an idempotent in $E(X,\t)$, then $(x,ux)\in {\bf
  P}$.
  \item If $x\in X$, then there is an minimal point $x'\in
  \overline{\O(x,\t)}$ such that $(x,x')\in {\bf P}$.
  \item If $(X,T)$ is minimal, then $(x,y)\in {\bf P}$ if and only
  if there is some minimal idempotent $u\in E(X,\t)$ such that
  $y=ux$.
\end{enumerate}

\end{prop}

The extension $\pi: (X, \t)\rightarrow (Y, \t)$ is called {\em
proximal} iff $R_\pi\subseteq {\bf P}$ iff ${\bf P}_\pi=\bigcap \{\t
\a\cap R_\pi: \a\in \U_X\}= R_\pi$. $\pi$ is {\em distal} if ${\bf
P}_\pi=\Delta_X$. $\pi$ is a {\em highly proximal} (HP) extension if
for every closed subset $A$ of $X$ with $\pi(A) = Y$, necessarily $A
= X$. It is easy to see that a HP extension is proximal. In the
metric case an extension $\pi: (X,T)\rightarrow (Y,T)$ of minimal
systems is HP iff it is an {\em almost 1-1 extension}, that is the
set $\{y\in Y: $ $\pi^{-1}(y)$ is a singleton $\}$ is a dense
$G_\delta$ subset of $Y$.

An extension $\pi: X\rightarrow Y$ of systems is called {\em
equicontinuous} or {\em almost periodic} if for every $\a\in \U_X$
there is $\b\in \U_X$ such that $\t \a\cap R_\pi\subseteq \b$.

In the metric case an equicontinuous extension is also called an
{\em isometric extension}. The extension $\pi$ is a {\em weakly
mixing extension} when $(R_\pi, \t)$ as a subsystem of the product
system $(X\times X, \t)$ is transitive.

\subsection{Vietoris topology and circle operation}

Let $2^X$ be the collection of nonempty closed subsets of $X$
endowed with the Vietoris topology. Note that a base for the
Vietoris topology on $2^X$ is formed by the sets
$$<U_1,U_2,\cdots,U_n>=\{A\in 2^X: A\subseteq \bigcup_{i=1}^nU_i\
\text{and $A\cap U_i\neq \emptyset$ for every $i$}\},$$ where $U_i$
is open in $X$. Then $(2^X, \t)$ defined by $tA=\{ta:a\in A\}$ is a
system again, and $S_\t$ acts on $2^X$ too. To avoid ambiguity we
denote the action of $S_\t$ on $2^X$ by the {\em circle operation}
as follows. Let $p\in S_\t$ and $D\in 2^X$, then define $p\c
D=\lim_{2^X} t_i D$ for any net $\{t_i\}_i$ in $\t$ with $t_i\to p$.
Moreover
\begin{equation*}
    p\c D=\{x\in X: \text{there are $d_i\in D$ with $x=\lim_i t_id_i$}\}
\end{equation*}
for any net $t_i\to p$ in $S_\t$. We always have $pD\subseteq p\c
D$.

\subsection{Ellis group}
The group of automorphisms of $(\M,\t)$, $G = {\rm Aut} (\M,\t)$ can
be identified with any one of the groups $u\M$ ($u\in J$) as
follows: with $\a\in uM$ we associate the automorphism $\hat{\a}:
(\M,\t)\rightarrow (\M,\t)$ given by right multiplication
$\hat{\a}(p)=p\a, p\in \M$. The group $G$ plays a central role in
the algebraic theory. It carries a natural $T_1$ compact topology,
called by Ellis the {\em $\tau$-topology}, which is weaker than the
relative topology induced on $G = u\M$ as a subset of $\M$.

\medskip

It is convenient to fix a minimal left ideal $\M$ in $S_\t$ and an
idempotent $u\in  \M$. As explained above we identify $G$ with $u\M$
and for any subset $A \subseteq G$, {\em $\tau$-topology} is
determined by
$$\cl_\tau A=u(u\c A)=G\cap (u\c A).$$
Also in this way we can consider the ``action" of $G$ on every
system $(X,\t)$ via the action of $S_\t$ on X. With every minimal
system $(X,T)$ and a point $x_0\in uX=\{x\in X: ux=x\}$ we associate
a $\tau$-closed subgroup $$\mathfrak{G}(X,x_0)=\{\a\in G: \a
x_0=x_0\}$$ the {\em Ellis group} of the pointed system $(X, x_0)$.

For a homomorphism $\pi: X\rightarrow Y$ with $\pi(x_0)=y_0$ we have
$$\mathfrak{G}(X, x_0)\subseteq\mathfrak{G}(Y, y_0).$$
It is easy to see that $u\pi^{-1}(y_0)=\mathfrak{G}(Y, y_0) x_0$.

\medskip

For a $\tau$-closed subgroup $F$ of $G$ the derived group $H(F)=F'$
is given by:
$$H(F)=F'= \bigcap\big\{ \cl_\tau O : O \ \text{is a
$\tau$-open neighborhood of $u$ in $F$ }\big \}.$$ $H(F)$ is a
$\tau$-closed normal subgroup of $F$ and it is characterized as the
smallest $\tau$-closed subgroup $H$ of $F$ such that $F/H$ is a
compact Hausdorff topological group. In particular, for an abelian
$\t$, the topological group $G/ H(G)$ is the {\em Bohr
compactification} of $\t$.

\subsection{Structure of minimal systems}

Let $\pi: (X,\t)\rightarrow (Y,\t)$ be a homomorphism of minimal
systems with $x_0\in X$ and $y_0=\pi(x_0)\in Y$. We say that $\pi$
is a {\em RIC} (relatively incontractible) extension if for every $y
= py_0\in Y$, $p$ an element of $\M$,
$$\pi^{-1}(y)=p\c u\pi^{-1}(y_0)=p\c Fx_0,$$ where $F =
\mathfrak{G}(Y, y_0)$. One can show that the extension $\pi : X \to
Y$ is RIC if and only if it is open and for every $n \ge 1$ the
minimal points are dense in the relation $R^n_\pi$. Note that every
distal extension is RIC. It then follows that every distal extension
is open.

\medskip

We say that a minimal system $(X, \t)$ is a {\em strictly PI system}
if there is an ordinal $\eta$ (which is countable when $X$ is
metrizable) and a family of systems
$\{(W_\iota,w_\iota)\}_{\iota\le\eta}$ such that (i) $W_0$ is the
trivial system, (ii) for every $\iota<\eta$ there exists a
homomorphism $\phi_\iota:W_{\iota+1}\to W_\iota$ which is either
proximal or equicontinuous (isometric when $X$ is metrizable), (iii)
for a limit ordinal $\nu\le\eta$ the system $W_\nu$ is the inverse
limit of the systems $\{W_\iota\}_{\iota<\nu}$,  and (iv)
$W_\eta=X$. We say that $(X,\t)$ is a {\em PI-system} if there
exists a strictly PI system $\tilde X$ and a proximal homomorphism
$\theta:\tilde X\to X$.

If in the definition of PI-systems we replace proximal extensions by
almost one-to-one extensions (or by highly proximal extensions in
the non-metric case) we get the notion of HPI {\em systems}. If we
replace the proximal extensions by trivial extensions (i.e.\ we do
not allow proximal extensions at all) we have I {\em systems}. These
notions can be easily relativized and we then speak about I, HPI,
and PI extensions.

\begin{thm}[Furstenberg]
A metric minimal system is distal if and only if  it is an I-system.
\end{thm}

\begin{thm}[Veech]
A metric minimal dynamical system is point distal if and only if  it
is an HPI-system.
\end{thm}

Finally we have the structure theorem for minimal systems, which we
will state in its relative form (Ellis-Glasner-Shapiro \cite{EGS},
Veech \cite{V77}, and Glasner \cite{G76}).

\begin{thm}[Structure theorem for minimal systems]\label{structure}
Given a homomorphism $\pi: X \to Y$ of minimal dynamical system,
there exists an ordinal $\eta$ (countable when $X$ is metrizable)
and a canonically defined commutative diagram (the canonical
PI-Tower)
\begin{equation*}
\xymatrix
        {X \ar[d]_{\pi}             &
     X_0 \ar[l]_{{\theta}^*_0}
         \ar[d]_{\pi_0}
         \ar[dr]^{\sigma_1}         & &
     X_1 \ar[ll]_{{\theta}^*_1}
         \ar[d]_{\pi_1}
         \ar@{}[r]|{\cdots}         &
     X_{\nu}
         \ar[d]_{\pi_{\nu}}
         \ar[dr]^{\sigma_{\nu+1}}       & &
     X_{\nu+1}
         \ar[d]_{\pi_{\nu+1}}
         \ar[ll]_{{\theta}^*_{\nu+1}}
         \ar@{}[r]|{\cdots}         &
     X_{\eta}=X_{\infty}
         \ar[d]_{\pi_{\infty}}          \\
        Y                 &
     Y_0 \ar[l]^{\theta_0}          &
     Z_1 \ar[l]^{\rho_1}            &
     Y_1 \ar[l]^{\theta_1}
         \ar@{}[r]|{\cdots}         &
     Y_{\nu}                &
     Z_{\nu+1}
         \ar[l]^{\rho_{\nu+1}}          &
     Y_{\nu+1}
         \ar[l]^{\theta_{\nu+1}}
         \ar@{}[r]|{\cdots}         &
     Y_{\eta}=Y_{\infty}
    }
\end{equation*}
where for each $\nu\le\eta, \pi_{\nu}$ is RIC, $\rho_{\nu}$ is
isometric, $\theta_{\nu}, {\theta}^*_{\nu}$ are proximal and
$\pi_{\infty}$ is RIC and weakly mixing of all orders. For a limit
ordinal $\nu ,\  X_{\nu}, Y_{\nu}, \pi_{\nu}$ etc. are the inverse
limits (or joins) of $ X_{\iota}, Y_{\iota}, \pi_{\iota}$ etc. for
$\iota < \nu$. Thus $X_\infty$ is a proximal extension of $X$ and a
RIC weakly mixing extension of the strictly PI-system $Y_\infty$.
The homomorphism $\pi_\infty$ is an isomorphism (so that
$X_\infty=Y_\infty$) if and only if  $X$ is a PI-system.
\end{thm}

\section{Proof of Theorem \ref{wm-extension}}

First we need the so-called {\em Ellis trick} in \cite{G76}. Refer
to \cite[Lemma X.6.1]{G76} for the proof. See \cite{Gl05} for more
discussions about weakly mixing extensions. Recall that $\M$ is the
universal minimal set.

\begin{lem}[Ellis trick]\label{str-lem4-10}
Let $F$ be $\tau$ closed subgroup of $G$ acting on $\M$ by right
multiplication, $\M\times F\rightarrow \M, (p,\a)\mapsto p\a$.
\begin{enumerate}
    \item there is a minimal idempotent $\w \in
J(\M)\cap \overline{F}$ such that $\overline{\w F}$ is $F$-minimal.

    \item if $V$ is a open subset of $\overline{wF}$, then ${\rm int}_{\tau}\cl_{\tau} (V\cap wF)\neq
    \emptyset$.
\end{enumerate}
\end{lem}

\begin{lem}\label{str-thm4-12}
Let $\pi : (X, \t)\rightarrow (Y, \t)$ be a RIC weakly mixing
extension of minimal systems and $u\in J(\M)$ be a minimal
idempotent. Let $x\in uX$, $y=\pi(x)$. Then for all $n\ge 2$, any
nonempty open subset $U$ of $\overline{u\pi^{-1}(y)}$ and any
transitive point $x'=(x_1', \cdots, x_{n-1}')\in R^{n-1}_\pi$ with
$\pi(x_j')=y, j=1,\cdots, n-1$, we have $\overline{\t (\{x'\}\times
U )}=R^n_\pi$.
\end{lem}

\begin{proof}
Note that we have $H(F)A=F$, where $F=\mathfrak{G} (Y, y),
A=\mathfrak{G}(X,x)$, since $\pi$ is weakly mixing.

\noindent {\em Claim:}
$$\{ux'\} \times \pi^{-1}(y)\subset \overline{\t(\{x'\}\times U )}.$$

\noindent {\em Proof of The Claim: } Set $V=\{p\in \ov{F}: px \in
U\}$. Then $V$ is a nonempty open set of $\ov{F}$ and by Ellis trick
we have $\widetilde{V}={\rm int}_\tau\cl_\tau (V\cap F)\neq
\emptyset$. By the definition of $H(F)$, there exists $\a\in F$ such
that $\a H(F)\subseteq\cl_\tau \widetilde{V}$.

Since $F=AH(F)=H(F)A$, we have
\begin{eqnarray*}
\ov{\t(\{x'\}\times U)} &\supseteq & u\c ( \{x'\}\times  U)\supseteq u\c ( \{x'\}\times Vx) \\
   &\supseteq &  \{u x'\} \times u(u\c V)x\supseteq  \{u x'\} \times u(u\c (V\cap
   F))x\\
   & =&  \{u x'\}\times  \cl_\tau (V\cap F)x\supseteq  \{u x'\}\times  \cl_\tau
   \widetilde{V}x\\
   &\supseteq & \{u x'\}\times \a H(F) x = \{u x'\}\times \a H(F) Ax\\
   & = & \{u x'\}\times \a F x =  \{u x'\}\times F x .
\end{eqnarray*}
Since $\pi$ is RIC, we have $u\c F x=\pi^{-1}(y)$. Hence
$$\ov{\t ( \{x'\}\times U)}\supseteq u\c ( \{ux'\}\times Fx)= \{ux'\}\times \pi^{-1}(y).$$
This ends the proof of the claim.

Now it is easy to see that $\overline{\t(\{x'\}\times U )}=R^n_\pi$.
Let $(x_1, x_2)\in R^n_\pi$, where $x_1\in R^{n-1}_\pi$. Since $x'$
is a transitive point of $R^{n-1}_\pi$, there exists a $p\in S_\t$
such that $px'=x_1$. Then $x_2 \in \pi^{-1}(py)=p\c \pi^{-1}(y)$.
Thus
$$(x_1, x_2)\in \{px'\}\times  p\c \pi^{-1}(y)\subseteq \overline{\t(
\{ux'\}\times \pi^{-1}(y))}\subseteq \overline{\t (\{x'\}\times U
)}.$$ Thus we have $R^n_\pi= \overline{\t (\{x'\}\times U )}$.
\end{proof}

\begin{thm}
Let $\pi : (X, \t)\rightarrow (Y, \t)$ be a RIC weakly mixing
extension of minimal metric systems and $y\in Y$. Then for all $n\ge
1$, there exists a transitive point $(x_1,x_2,\ldots,x_n)$ of
$R^n_\pi$ with $x_1, x_2, \ldots , x_n\in \pi^{-1}(y)$.
\end{thm}

\begin{proof}
It is obvious for the case when $n=1$, since $R^1_\pi=X$. Now assume
it is true for $n-1$ ($n\ge 2$). Fix a transitive point
$x'=(x_1,x_2,\ldots,x_{n-1})\in R^{n-1}_\pi$ with
$x_1,x_2,\ldots,x_{n-1}\in \pi^{-1}(y)$. Assume that $y\in uY$ for
some minimal idempotent $u\in J(\M)$.

For each $\ep >0$, define
$$V_\ep=\{x\in \overline{u \pi^{-1}(y)}: \t (x',x)\ \text{is
$\ep$-dense in }\ R^n_\pi\}.$$ It is easy to verify that $V_\ep$ is
open. Now we show that $V_\ep$ is dense in
$\overline{u\pi^{-1}(y)}$. For any $\Lambda\subseteq X^n, z\in X^n,
\d >0$, $\Lambda \stackrel{\d} \sim z$ is defined by $\rho(z,
z')<\d, \forall z'\in \Lambda$.

Now let $\{z_1, z_2,\cdots, z_n\}$ be an $\ep$-net of $R^n_\pi$,
i.e. for each $z\in R_\pi^n$ there is some $z_j$ ($j\in
\{1,2,\ldots, n\}$) such that $\rho(z,z_j)<\ep$. Let $U$ be an open
subset of $\overline{w\pi^{-1}(y)}$. By Lemma \ref{str-thm4-12},
$\overline{\t ( \{x'\}\times U)}=R^n_\pi$. So there are some open
subset $U_1\supseteq U$ and $t_1\in \t$ such that $t_1( \{x'\}\times
U_1)\stackrel{\ep}\sim z_1$. Again, by Lemma \ref{str-thm4-12},
$\overline{\t ( \{x'\}\times U_1)}=R^n_\pi$. So there are an open
subset $U_2\supseteq U_1$ and $t_2\in \t$ such that $t_2(
\{x'\}\times U_2)\stackrel{\ep}\sim z_2$. $\ldots$ Inductively, we
have a sequence $U_1\supseteq U_2\supseteq \cdots \supseteq U_n$
(relatively open) and $t_1, \ldots, t_n\in \t$ such that $t_j(
\{x'\}\times U_n)\stackrel{\ep}\sim z_j, \forall j\in \{1,2, \ldots,
n\}$. Hence $U_n\subseteq V_\ep$. This means that $V_\ep$ is dense
in $\overline{u \pi^{-1}(y)}$.

Let $\Gamma=\bigcap_{n=1}^\infty V_{1/n}$. Then $\Gamma$ is a
residual set of $\overline{u \pi^{-1}(y)}$, and for all $x\in
\Gamma$, we have $\overline{\t (x',x)}=R^n_\pi$. In particular,
there exists a transitive point $(x_1,x_2,\ldots,x_n)$ of $R^n_\pi$
with $x_1, x_2, \ldots , x_n\in \pi^{-1}(y)$. The proof is
completed.
\end{proof}


\end{document}